\documentclass[11pt]{article}
\usepackage{amsfonts}
\usepackage{indentfirst}
\usepackage{graphicx}
\usepackage{cite}
\usepackage{amsthm}
\usepackage{amsmath}
\usepackage{epstopdf}
\usepackage{amssymb}
\usepackage{mathrsfs}
\usepackage[all]{xy}
\usepackage{subfigure}
\usepackage{psfrag}
\usepackage{array}
\usepackage{float}%%%byjhl
\usepackage{txfonts}
\usepackage{enumitem}

\usepackage[colorlinks,
            pdfborder=001,
            linkcolor=blue,
            anchorcolor=green,
            citecolor=red
            ]{hyperref}%%%byjhl

\newtheorem{Remark}{{Remark}}[section]
\newtheorem{Theorem}{Theorem}[section]
\newtheorem{Lemma}{Lemma}[section]

\numberwithin{equation}{section}

\newcommand\bes{\begin{eqnarray}} \newcommand\ees{\end{eqnarray}}
\newcommand{\bess}{\begin{eqnarray*}}
\newcommand{\eess}{\end{eqnarray*}}

\newcommand\qq{\eqref}
\newcommand\lm{\lambda}

\begin{document}
\title{{\bf\Large Global bifurcation of coexistence states for a prey-taxis
system with homogeneous Dirichlet boundary conditions}
\footnote{The first author's work was supported by NSFC Grant 11901446 and the
Postdoctoral Science Foundation of China (2021T140530). The second author's work
was supported by NSFC Grant 12171120.}
}
\author{Shanbing Li$^1$,~~Mingxin Wang$^{2,}$
\thanks{The correspondence author. E-mail address:
mxwang@hpu.edu.cn (M. Wang)}
\\
{\small $^1$\ School of Mathematics and Statistics, Xidian University, Xi'an,
710071, PR China}\\
{\small $^2$ School of Mathematics and Information Science, Henan Polytechnic
University, Jiaozuo, 454000, PR China}\\
}
\date{}        % Enter your date or \today between curly braces
\maketitle
\begin{center}
\begin{minipage}{135mm}
\noindent{\bf Abstract}--This paper is concerned with positive solutions of boundary value problems
\begin{equation*}
\left\{\begin{array}{ll}
{\rm div}\left(d(v)\nabla u-u\chi(v)\nabla v\right)+\lambda u-u^2
+\gamma u F(v)=0,  & x \in \Omega,\\[1mm]
D \Delta v+\mu v-v^2-u F(v)=0,  & x \in \Omega,\\[1mm]
u=v=0,  & x \in \partial \Omega.
\end{array}\right.
\end{equation*}
This is the stationary problem associated with the predator-prey system
with prey-taxis, and $u$ (resp. $v$) denotes the population density of
predator (resp. prey). In particular, the presence of $\chi(v)$ represents
the tendency of predators to move toward the increasing preys gradient
direction. Regarding $\lambda$ as a bifurcation parameter, we make a
detailed description for the global bifurcation structure of the set of
positive solutions. So that ranges of parameters are found for which the
system admits positive solutions.

\vskip 4mm\noindent{\bf Keywords}: Prey-taxis; Quasilinear elliptic systems;
Boundary value problems; Global bifurcation.

\vskip 4mm\noindent{\bf Mathematics Subject Classification}: 35J57,~~35J60,~~92D25
\end{minipage}
\end{center}

\section{Introduction and main results}
{\setlength\arraycolsep{2pt}

In the last twenty years, predator-prey systems with prey-taxis
proposed by Kareiva and Odell \cite{ko1987} have attracted great
attention in mathematical analysis. The introduction of prey-taxis term
into classical reaction-diffusion systems allows the mathematical models
to capture much more important features of phenomena in ecology.
Meanwhile, the presence of prey-taxis term causes enormous difficulties
in the analytical treatment. Many techniques which worked successfully
for classical reaction-diffusion systems are no longer applicable, and
many fundamental questions are left open.

Under homogeneous Neumann boundary conditions, predator-prey systems with
prey-taxis have been studied most extensively---perhaps because they
resemble attractive chemotaxis systems from a mathematical point of view,
which have been studied in comparatively great detail. One can refer to
\cite{ccw2020,jw2021,lhl2009,ww2019,wsw2016}
and references therein. However, under homogeneous Dirichlet boundary
conditions, predator-prey systems with prey-taxis has received less
attention. In fact, to the best of our knowledge, only Cintra et al. in
\cite{cms2018,cms2019} have analyzed the corresponding Dirichlet systems.

In the present article, we study positive steady-state solutions to the
following predator-prey system with prey-taxis under homogeneous Dirichlet
boundary conditions:
\begin{equation}\label{ys11}
\left\{\begin{array}{ll}
\partial_{t} u={\rm div}\left(d(v)\nabla u-u\chi(v)\nabla v\right)
+\lambda u-u^2+\gamma u F(v),  & x \in \Omega,~t>0, \\[1mm]
\partial_{t} v=D \Delta v+\mu v-v^2-u F(v),  & x \in \Omega,~t>0,\\[1mm]
u=v=0,  & x \in \partial \Omega,~t>0, \\[1mm]
u(x, 0)=u_0(x)\geq0, \quad v(x, 0)=v_0(x)\geq0,  & x \in \Omega,
\end{array}\right.
\end{equation}
where $\Omega$ is a bounded domain in $\mathbb{R}^n~(n\geq1)$ with smooth
boundary $\partial\Omega$. From the ecological viewpoint, unknown functions
$u=u(x,t)$ and $v=v(x,t)$ stand for the population densities of predator
and prey at position $x\in\Omega$ and time $t>0$, respectively;
constants $\lambda$ and $\mu$ denote the growth rates of each species,
where $\lambda\in \mathbb{R}$ and $\mu>0$; positive constant $\gamma$
accounts for the intrinsic predation rate; nonnegative function $F(v)$
represents the functional response of the predator; positive constant $D$
is the diffusion rate of the prey; the term ${\rm div}(d(v)\nabla u)$
describes the diffusion of the predator with coefficient $d(v)$; the term
$-{\rm div}(u\chi(v)\nabla v)$ describes the prey-taxis with coefficient
$\chi(v)$, which shows the tendency of predator moving toward the increasing
prey gradient direction. For more details on the backgrounds of this model,
one can refer to \cite{ko1987}.

The corresponding steady-state problem of (\ref{ys11}) is given by
\begin{equation}\label{ys12}
\left\{\begin{array}{ll}
-{\rm div}\left(d(v)\nabla u-u\chi(v)\nabla v\right)=\lambda u-u^2+
\gamma u F(v),  & x \in \Omega,\\[1mm]
-D \Delta v=\mu v-v^2-u F(v),  & x \in \Omega,\\[1mm]
%u\geq0,~~~~v\geq0,  & x \in \Omega,\\[1mm]
u=v=0,  & x \in \partial \Omega.
\end{array}\right.
\end{equation}
Specifically in the sequel, functions $d(v),\chi(v)$ and $F(v)$ are
assumed to fulfill the following conditions:

$(H_{d})$\; $d(v)\in C^2([0, \infty)),~d(v)>0$ and $d'(v)\le 0$ $[0, \infty)$;

$(H_{\chi})$\; $\chi(v)\in C^1([0, \infty)),~ \chi(v)\ge 0$ in $[0, \infty)$;

  $(H_{F})$\; $F(v)\in C^2([0, \infty)),~F(0)=0$ and $F(v)>0$ in $[0, \infty)$.\\
Obviously, the hypotheses on $F(v)$ imply that

%\begin{itemize}[itemsep= -2 pt,topsep = 0 pt, leftmargin = 30 pt]
%  \item $F^{\prime}(0)\geq 0$~~{\rm and}~~$F(v)=v \mathbb{F}(v)$,
%\end{itemize}
$$F^{\prime}(0)\geq 0~~~{\rm and}~~~F(v)=v \mathbb{F}(v),$$
where $\mathbb{F}(v)\in C^1([0, \infty))$ and $\mathbb{F}(v)>0$ in
$(0,\infty)$. The most widely used forms of $F(v)$ in the literature are:
\begin{equation*}
\begin{array}{ll}
F(v) = v~~~({\rm Lotka-Volterra~type}); &
~~~F(v) = \frac{v}{\zeta+v}~~~({\rm Holling~type~II}); \\[1mm]
F(v) = \frac{v^2}{\zeta+v^2}~~~({\rm Holling~type~III}); &
~~~F(v) = \frac{v}{\zeta+v^2}~~~({\rm Holling~type~IV}),
\end{array}
\end{equation*}
where $\zeta>0$ is a constant. Moreover, it is clear that Lotka-Volterra type,
Holling~type~II and Holling~type~III response functions are monotonic with
respect to $v$, while Holling~type~IV response function is nonmonotonic with
respect to $v$.

With the restriction $\chi(v)=-d^{\prime}(v)$, system \eqref{ys12} can be
written as
\begin{equation}\label{ys12*}
\left\{\begin{array}{ll}
\Delta\left(d(v) u\right)+\lambda u-u^2+\gamma uF(v)=0, & x \in \Omega, \\[1mm]
D \Delta v+ \mu v-{v}^2-uF(v)=0, & x \in \Omega, \\[1mm]
u=v=0, & x \in \partial\Omega.
\end{array}\right.
\end{equation}
A standard approach to deal with (\ref{ys12*}) is to apply the change of
variable $U=d(v)u$, which transforms (\ref{ys12*}) into a semilinear system,
decoupled in the diffusion (see, e.g., \cite{ln1996,y2008}). With the help
of this convenient change of variable, our previous paper \cite{lm2022}
established the existence of positive solutions of (\ref{ys12*}) by using
the theory of fixed point index in positive cones, and investigated the
limiting behaviors of positive solutions of (\ref{ys12*}) as $b$ goes to
$\infty$ in the case $d(v)=1+\frac{a}{1+bv}$, where $a,b$ are positive
constants. These results not only generalize the results obtained in
\cite{ky2009,wl2015}, but also present some new conclusions.

For the general case (without the restriction $\chi(v)=-d^{\prime}(v)$),
system \eqref{ys12} is more complicated and there is a few works. To our best
knowledge, \eqref{ys12} has only been studied in \cite{cms2018,cms2019} with
the Lotka-Volterra type response function (i.e., $F(v)=v$).

Before the formal statement of our main result, we introduce some notations
and basic facts. Let us denote by $\sigma_1(p,q;r)$ the principal eigenvalue of
\bes\left\{\begin{array}{ll}
-{\rm div}\left(p(x)\nabla \phi\right)+q(x)\phi=\sigma r(x)\phi,
~~ &x \in \Omega,\\[1mm]
\phi=0,~~ &x \in \partial\Omega,
\end{array}\right.\label{ys13}
 \ees
where $p(x)\in C^{1,\alpha}(\overline{\Omega})$, $q(x),r(x)\in C^{\alpha}
(\overline{\Omega})$ for some $\alpha\in(0,1)$, and $p(x)\geq p_0>0$ and
$r(x)>0$. In particular, when $p=r=1$ and $q=0$, we denote
 \[\sigma_1=\sigma_1(1,0;1).\]

For the fixed functions $p(x)\in C^{1,\alpha}(\overline{\Omega})$,
$b(x)\in C^{\alpha}(\overline{\Omega})$ with $\alpha\in(0,1)$,
$p(x)\geq p_0>0$ and $b(x)\geq b_0>0$, the boundary value problem
\bes\label{ys14}
\left\{\begin{array}{ll}
-{\rm div}\left(p(x)\nabla\phi\right)=a\phi-b(x)\phi^2,~~ &x \in \Omega,\\[1mm]
\phi=0,~~ &x \in \partial\Omega
\end{array}\right.
 \ees
has positive solution if and only if $a>\sigma_1(p,0;1)$, and the positive
solution is unique when it exists, denoted by $\theta_{p,a,b}$.
When $b(x)\equiv1$, we simply denote $\theta_{p,a}=\theta_{p,a,1}$. Thus,
system \eqref{ys12} admits two semitrivial solutions
$(u,v)=(\theta_{d(0),\,\lambda},0)$ and $(u,v)=(0,\theta_{D,\,\mu})$ when
$\lambda>d(0)\sigma_1$ and $\mu>D\sigma_1$, where $\theta_{d(0),\,\lambda}$
solves
\bes\label{1.6}
\left\{\begin{array}{ll}
-d(0)\Delta u=\lambda u-u^2,~~&x\in\Omega,\\[1mm]
u=0,~~ &x\in\partial\Omega,
\end{array}\right.
 \ees
and $\theta_{D,\,\mu}$ solves
\bes\label{ys16}
\left\{\begin{array}{ll}
-D\Delta v=\mu v-v^2,~~&x\in\Omega,\\[1mm]
v=0, &x\in\partial\Omega.
\end{array}\right.
 \ees
For the sake of brevity, we denote
 \[\theta_\lambda=\theta_{d(0),\,\lambda},\;\;\;\omega_\mu=\theta_{D,\,\mu}.\]
We denote two semitrivial solution branches with parameter $\lambda$ by
\begin{equation*}
\left\{\begin{array}{ll}
\Gamma_{u}=\{(\lambda,\theta_\lambda,0):
\lambda>d(0)\sigma_1\},  &  \\[1mm]
\Gamma_{v}=\{(\lambda,0,\omega_\mu): \lambda\in \mathbb{R}\} \;\;\text{with}\;\;\mu>D \sigma_1.  &
\end{array}\right.
\end{equation*}

Let $\Phi_\mu$ be the positive eigenfunction associated to the principal eigenvalue
\bes
 \lm_\mu=\sigma_1\left(d(\omega_\mu)
e^{g(\omega_\mu)},\, -\gamma F(\omega_\mu)
e^{g(\omega_\mu)};\, e^{g(\omega_\mu)}\right),
 \label{1.8a}\ees
where
$$g(\omega_\mu)=\int_0^{\omega_\mu}\frac{\chi(\tau)}{d(\tau)}
d\tau.$$
That is, $\Phi_\mu$ satisfies
\begin{equation}\label{1.9}
\left\{\begin{array}{ll}
-{\rm div}\left(d(\omega_\mu)e^{g(\omega_\mu)}\nabla
\Phi_\mu\right)-\gamma F(\omega_\mu)
e^{g(\omega_\mu)}\Phi_\mu=
\lambda_\mu e^{g(\omega_\mu)}
\Phi_\mu,& x\in\Omega,\\[1mm]
\Phi_\mu=0,& x\in\partial\Omega.
\end{array}\right.
\end{equation}
Denote
\begin{equation}
\left\{\begin{array}{ll}
\phi_\mu=e^{g(\omega_\mu)}\Phi_\mu>0,\\[1mm]
\psi_\mu=-(-D\Delta+2\omega_\mu-\mu)^{-1}
\left(F(\omega_\mu)e^{g(\omega_\mu)}\Phi_\mu\right).
\end{array}\right.
  \label{1.10}\end{equation}

Now we are able to state the main theorem of this paper.

\begin{Theorem}\label{th11}
Assume that $\mu>D\sigma_1$ is fixed. Then the following statements
hold.
\begin{itemize}[itemsep= -2 pt, topsep = 0 pt, leftmargin = 35 pt]
  \item [{\rm(1)}] There exists a positive constant $\underline{\lambda}=\underline{\lambda}(\mu)$
  such that \eqref{ys12} has no positive solution when $\lambda\leq \underline{\lambda}$.
\item [{\rm(2)}] Assume that $\min_{[0,\,\mu]}\mathbb{F}(v)>0$. Then there exists
  a positive constant $\overline{\lambda}=\overline{\lambda}(\mu)$ such that \eqref{ys12} has no positive
  solution for $\lambda>\overline{\lambda}$.
  \item [{\rm(3)}] If $F'(0)=0$, i.e., $\mathbb{F}(0)=0$, then
  \begin{itemize}[itemsep= -2 pt, topsep = 0 pt, leftmargin = 25 pt]
  \item [{\rm(3-1)}] $\lambda=\lambda_\mu$ is a bifurcation point where an unbounded
  continuum $\mathcal{C}_1$ of positive solutions to \eqref{ys12} bifurcates
  from the semitrivial solution branch $\Gamma_{v}$ at
  $(\lambda_\mu,0,\omega_\mu)$;
  \item [{\rm(3-2)}] near $(\lambda_\mu,0,\omega_\mu)$, $\mathcal{C}_1$ is a smooth curve $(\lambda(s),u(s),v(s))$ with $s\in(0,\varepsilon)$, such
  that $(\lambda(0),u(0),v(0))=(\lambda_\mu,0,\omega_\mu)$ and
  \bes\label{ys314}
   \lambda^{\prime}(0)
  &=& \left(\int_{\Omega}e^{g(\omega_\mu)}\Phi_\mu^2dx
     \right)^{-1}\left[\int_{\Omega}\left(d^{\prime}(\omega_\mu)
     +\chi(\omega_\mu)\right)e^{g\left(\omega_\mu\right)}
     \psi_\mu\left|\nabla\Phi_\mu\right|^2dx
     \right.\nonumber \\
  &&+\int_{\Omega}e^{2g(\omega_\mu)}\Phi_\mu^3dx
     -\lambda_\mu\int_{\Omega}\frac{\chi(\omega_\mu)}{d(\omega_\mu)}
     e^{g(\omega_\mu)}\psi_\mu\Phi_\mu^2dx\nonumber\\
  &&\left.-\gamma\int_{\Omega}\left(F^{\prime}(\omega_\mu)+
     \frac{\chi(\omega_\mu)}{d(\omega_\mu)}F(\omega_\mu)\right)
     e^{g(\omega_\mu)}\psi_\mu\Phi_\mu^2dx\right];
  \ees
  \item [{\rm(3-3)}] the projection of $\mathcal{C}_1$ on $\lambda$-axis:
  $\mathrm{Proj}_{\lambda}\mathcal{C}_1\supset(\lambda_\mu,\infty)$,
  and so \eqref{ys12} has at least one positive solution when
  $\lambda>\lambda_\mu$;
  \item [{\rm(3-4)}] the semitrivial solution $(\theta_\lambda,0)$ is an
  unstable steady state of \eqref{ys12} for any $\lambda>d(0)\sigma_1$,
  and there is no bifurcation of positive solutions occurring along the
  semitrivial solution branch $\Gamma_{u}$.
  \end{itemize}
  \item [{\rm(4)}] If $F'(0)>0$, i.e., $\mathbb{F}(0)>0$, then
  \begin{itemize}[itemsep= -2 pt, topsep = 0 pt, leftmargin = 25 pt]
  \item [{\rm(4-1)}] there exists a unique $\lambda^*=\lambda^*(\mu)$
  determined by $\mu=\mu_{\lambda^*}$, where
  $$\mu_\lambda=\sigma_1(D,\,\theta_\lambda F'(0);\,1),$$
  such that the semitrivial solution $(\theta_\lambda,0)$ is a
  stable steady state of \eqref{ys12} for $\lambda>\lambda^*$ and
  it is an unstable steady state of \eqref{ys12} for $\lambda<\lambda^*$;
  \item [{\rm(4-2)}] $\lambda^*(\mu)$ is strictly increasing with respect to
  $\mu$, and
  $$\lim_{\mu\rightarrow D \sigma_1}\lambda^*(\mu)=
  d(0) \sigma_1,~~~{\rm and}~~~\lim_{\mu\rightarrow\infty}
  \lambda^*(\mu)=\infty;$$
  \item [{\rm(4-3)}] $\lambda=\lambda^*$ is a bifurcation point where a continuum
  $\mathcal{C}_2$ of positive solutions to \eqref{ys12} bifurcates from the
  semitrivial solution branch $\Gamma_{u}$ at
  $(\lambda^*,\theta_{\lambda^*},0)$, and near
  $(\lambda^*,\theta_{\lambda^*},0)$, $\mathcal{C}_2$ is a smooth
  curve $(\lambda(s),u(s),v(s))$ with $s\in(0,\varepsilon)$ such that
  $(\lambda(0),u(0),v(0))=(\lambda^*,\theta_{\lambda^*},0)$;
  \item [{\rm(4-4)}] $\mathcal{C}_2$ can be extended to a bounded global continuum
  of positive solutions to \eqref{ys12}, which meets the semitrivial solution
  branch $\Gamma_{v}$ at $(\lambda_\mu,0,\omega_\mu)$, and so \eqref{ys12}
  has at least one positive solution for $\lambda\in(\lambda_\mu,\lambda^*)$
  or $\lambda\in(\lambda^*,\lambda_\mu)$.
  \end{itemize}
\end{itemize}
\end{Theorem}

The rest of this paper is arranged as follows. In section 2, we establish
the $C(\overline{\Omega})$-estimate and $W^{2,p}(\Omega)$-estimate for
 positive solutions of \eqref{ys12}. In section 3, we complete the proof of
Theorem \ref{th11}. The significance of current studies is outlined and some
further works are presented in section 4.

\section{A priori estimates}

In this section, we will show some a priori estimates for positive solutions
of \eqref{ys12}. It is well-known that the priori estimates are the fundamental
in yielding the existence and nonexistence of positive solution of \eqref{ys12}.

We first prove the $C(\overline{\Omega})$-estimate for positive solutions
of \eqref{ys12}.

\begin{Lemma}\label{le21}\; If $\mu\leq D \sigma_1$, then \eqref{ys12} has no positive solution. If $\mu> D\sigma_1$, then positive solutions  $(u,v)$ of \eqref{ys12} satisfy
\begin{equation*}
\max_{\overline{\Omega}}u\leq e^{\int_0^{\mu}\frac{\chi(s)}{d(s)}ds}
\left(|\lambda|+\gamma\max\limits_{[0,\,\mu]}F(v)\right),~~~~
\max_{\overline{\Omega}}v\leq \mu.
\end{equation*}
\end{Lemma}

\begin{proof} Let $(u,v)$ be a positive solution of \eqref{ys12}. Multiplying \eqref{ys12} by $v$ and integrating the result by parts, we have
\begin{equation*}
D\int_{\Omega}|\nabla v|^2dx=\int_{\Omega}\left(\mu v-v^2-uF(v)\right)vdx
<\mu\int_{\Omega}v^2dx.
\end{equation*}
It is deduced that
   $$D\sigma_1\int_{\Omega}v^2dx<\mu\int_{\Omega}v^2dx$$
by the Poincar\'{e} inequality:
$$\sigma_1\int_{\Omega}v^2dx\leq
\int_{\Omega}|\nabla v|^2dx,\;\;\forall\; v\in H^1_0(\Omega).$$
Thus, $\mu>D\sigma_1$ since $v>0$ in $\Omega$.

Now we assume that $\mu> D\sigma_1$, and $(u,v)$ is a positive solution of
\eqref{ys12}. Let $v(x_1)=\max_{\overline{\Omega}}v(x)$. Then $x_1\in\Omega$
and
\begin{equation*}
0\leq-D\Delta v(x_1)=v(x_1)\left(\mu-v(x_1)-u(x_1)
\mathbb{F}(v(x_1))\right).
\end{equation*}
Hence, $\max_{\overline{\Omega}}v(x)=v(x_1)\leq \mu$.

Let $w=e^{-g(v)}u$ with
 $$g(v)=\int_0^{v}\frac{\chi(\tau)}{d(\tau)}d\tau.$$
Then, by the equations of $u$,
\begin{equation}\label{ys21}
\left\{\begin{array}{ll}
-{\rm div}\left(d(v)e^{g(v)}\nabla w\right)=e^{g(v)}w\left(\lambda
-e^{g(v)}w+\gamma F(v)\right), & x\in\Omega,\\[1mm]
w=0, & x\in\partial\Omega.
\end{array}\right.
\end{equation}
Let $w(x_2)=\max_{\overline{\Omega}}w(x)$. Then
\begin{equation*}
\left.{\rm div}\left(d(v)e^{g(v)}\nabla w\right)\right|_{x=x_2}
=\left.\nabla\left(d(v)e^{g(v)}\right)\cdot\nabla w\right|_{x=x_2}
\left.+d(v)e^{g(v)}\Delta w\right|_{x=x_2}\leq 0.
\end{equation*}
It follows from (\ref{ys21}) that
\begin{equation*}
w(x_2)\leq e^{-g(v(x_2))}\left(\lambda+\gamma F(v(x_2))\right)\leq
|\lambda|+\gamma\max_{[0,\,\mu]}F(v).
\end{equation*}
Since $w=e^{-g(v)}u$, we have
\begin{equation*}
u(x)=e^{g(v(x))}w(x)\leq e^{g(\mu)}w(x_2)\leq e^{\int_0^{\mu}
\frac{\chi(s)}{d(s)}ds}\left(|\lambda|+\gamma\max_{[0,\,\mu]}F(v)\right)
\end{equation*}
for any $x\in\Omega$. This gives the desired estimate for $u$ in $\Omega$.
\end{proof}

We next prove the $W^{2,p}(\Omega)$-estimate for positive solutions of
\eqref{ys12}.

\begin{Lemma}\label{le22} Let $\mu>D\sigma_1$ be fixed. Then for any $p\in(1,\infty)$,
there exists a positive number $M$, depending on the parameters of
\eqref{ys12}, such that any positive solution $(u,v)$ of \eqref{ys12} satisfies
\begin{equation*}
\|(u,v)\|_{W^{2,p}(\Omega)}\leq M.
\end{equation*}
\end{Lemma}

\begin{proof}
In this proof, we shall use $M_{i}$ to denote generic positive constants
depending on the parameters of system \eqref{ys12}. By Lemma \ref{le21},
$\|\mu v-v^2-u F(v)\|_{L^\infty(\Omega)}\leq M_1$. It follows from the $L^p$ theory that $\|v\|_{W^{2,p}(\Omega)}\leq M_2$. Then the Sobolev embedding
theorem yields
\begin{equation}\label{ys23}
\|v\|_{C^1(\overline{\Omega})}\leq M_3.
\end{equation}
Likewise, apply the above analysis to \eqref{ys21} to obtain
\begin{equation}\label{ys24}
\|e^{-g(v)}u\|_{C^1(\overline{\Omega})}=\|w\|_{C^1(\overline{\Omega})}
\leq M_{4}
\end{equation}
since $e^{g(v)}w(\lambda-e^{g(v)}w+\gamma F(v))\in L^\infty(\Omega)$ and $d(\mu)\leq d(v)e^{g(v)}\leq d(0)e^{g(\mu)}$. Noticing that
\begin{equation*}
\nabla w=\nabla\left(e^{-g(v)}u\right)=e^{-g(v)}\nabla u
-\frac{\chi(v)}{d(v)}e^{-g(v)}u\nabla v.
\end{equation*}
It follows from Lemma \ref{le21} that, by use of monotonicity of $d$ and $g$,
\begin{equation}
 e^{-g(\mu)}|\nabla u|\leq\left|e^{-g(v)}\nabla u\right|\leq
\left|\nabla\left(e^{-g(v)}u\right)\right|+
\left|\frac{\chi(v)}{d(v)}e^{-g(v)}u\nabla v\right|\leq \left|\nabla\left(e^{-g(v)}u\right)\right|+
\frac{\max_{[0,\,\mu]}\chi(v)}{d(\mu)}|u\nabla v|.
\label{ys25}\end{equation}
In accordance with Lemma \ref{le21} and \eqref{ys23}, we have $|u\nabla v|\leq M_{5}$.
This, combining with \qq{ys23} and \qq{ys25}, allows us to derive
\begin{equation}\label{ys27}
\left\|u\right\|_{C^1(\overline{\Omega})}\leq M_{6}.
\end{equation}

Rewrite the equations of $u$ as
\begin{equation*}
\left\{\begin{array}{ll}
-\Delta u=\frac1{d(v)}\left(\left(d'(v)-\chi(v)\right)\nabla u\cdot\nabla v
-u\chi'(v)|\nabla v|^2-u\chi(v)\Delta v+\lambda u-u^2+\gamma u F(v)\right),
& x\in\Omega,\\[1mm]
u=0, & x\in\partial\Omega.
\end{array}\right.
\end{equation*}
{In view of \eqref{ys23}, \eqref{ys27}, and the hypotheses $(H_{d})$,
$(H_{\chi})$ and $(H_{F})$, it follows that}
\begin{equation*}
\left\|\frac1{d(v)}\left(\left(d'(v)-\chi(v)\right)\nabla u\cdot\nabla v
-u\chi'(v)|\nabla v|^2-u\chi(v)\Delta v+\lambda u-u^2+\gamma u F(v)\right)
\right\|_{L^{p}(\Omega)}\leq M_{7}.
\end{equation*}
Then the $L^p$ theory gives the estimate of $\|u\|_{W^{2,p}(\Omega)}$.
\end{proof}

\section{Proof of Theorem \ref{th11}}

In this section, we will show our main theorem. We first state some
properties of the linear eigenvalue problem (\ref{ys13}) (see, e.g.,
 \cite[Lemma 2.2]{cms2018}) and the diffusive logistic equation (\ref{ys14})
(see, e.g., \cite[Lemma 2.3]{cms2018}).

\begin{Lemma}\label{le31}\; The principal eigenvalue $\sigma_1(p, q; r)$  of \eqref{ys13} satisfies
\begin{equation*}
\sigma_1(p, q; r)=
\inf_{\psi\in H^1_0(\Omega),\psi\neq0}\frac{\int_{\Omega}p(x)
|\nabla \psi|^2dx+\int_{\Omega}q(x)\psi^2 dx}{\int_{\Omega}r(x)\psi^2 dx}.
\end{equation*}
Moreover,
\begin{itemize}[itemsep= -2 pt,topsep = 0 pt, leftmargin = 25 pt]
  \item [(1)] $\sigma_1(p, q; r)$ is increasing
  with respect to $p$;

\item [(2)] $\sigma_1(p, q; r)$ is increasing
  with respect to $q$;

\item [(3)] the monotonicity of
  $\sigma_1(p, q; r)$ with respect to $r$
  depends on the sign of $\sigma_1(p,q;1)$, and
  \begin{itemize}[itemsep= -2 pt,topsep = -2 pt, leftmargin = 18 pt]
  \item [(a)] if $\sigma_1(p,q;1)>0$, then
  $\sigma_1(p, q; r)$ is positive and decreasing
  with respect to $r$;

\item [(b)] if $\sigma_1(p,q;1)=0$, then  $\sigma_1(p, q; r)=0$ for every $r$;

\item [(c)] if $\sigma_1(p,q;1)<0$, then  $\sigma_1(p, q; r)$ is negative and increasing with respect to $r$.
  \end{itemize}
\end{itemize}
\end{Lemma}

\begin{Lemma}\label{le32}
The diffusive logistic equation (\ref{ys14}) admits a unique positive solution,
denoted by $\theta_{p,a,b}$, if and only if
\begin{equation*}
a>\sigma_1(p,0;1).
\end{equation*}
Moreover, the mapping
$$a\in(\sigma_1(p,0;1),\,\infty)
\rightarrow\theta_{p,a,b}\in C_0^2(\overline{\Omega})$$
is continuous and increasing. Furthermore, $\theta_{p,a,b}$ satisfies
\begin{equation*}
\frac{a-\sigma_1(p,0;1)}
{\|b\|_{C(\overline{\Omega})}\|\phi_a\|_{C(\overline{\Omega})}}
\phi_a(x)\leq \theta_{p,a,b}(x)\leq \frac{a}{b_0},\;\;x\in\Omega,
\end{equation*}
where $\phi_a$ is the positive eigenfunction associated to $\sigma_1(p,0;1)$.
\end{Lemma}

\subsection{Proof of Theorem \ref{th11}(1)}

Let $(u,v)$ be any positive solution of \eqref{ys12}. In view of (\ref{ys21}),
$w=e^{-g(v)}u$ satisfies
\begin{equation*}
\left\{\begin{array}{ll}
-{\rm div}\left(d(v)e^{g(v)}\nabla w\right)+e^{g(v)}(u-\gamma F(v))w
=\lambda e^{g(v)}w, \;\;& x\in\Omega,\\[1mm]
w=0, & x\in\partial\Omega.
\end{array}\right.
\end{equation*}
As $w>0$ in $\Omega$, it follows that
\begin{equation*}
\lambda=\sigma_1\left(d(v)e^{g(v)},\,e^{g(v)}(u-\gamma F(v));\,e^{g(v)}\right).
\end{equation*}
We further derive from Lemma \ref{le21} and Lemma \ref{le31} that
\begin{equation*}
\lambda>\sigma_1\left(d(\mu),\,-\gamma e^{g(\mu)}\max_{[0,\,\mu]}F(v);\,e^{g(v)}\right).
\end{equation*}
According to the sign of $\sigma_1(d(\mu),\;-\gamma
e^{g(\mu)}\max_{[0,\,\mu]}F(v);\,1)$, our applying Lemma \ref{le31} gives
\begin{equation*}
\lambda>\left\{\begin{array}{ll}
\sigma_1(d(\mu),\,-\gamma e^{g(\mu)}\max_{[0,\,\mu]}F(v);\,
e^{g(\mu)})\;\;  &  {\rm if}~~\sigma_1(d(\mu),\;-\gamma
e^{g(\mu)}\max_{[0,\,\mu]}F(v);\,1)>0,\\[1mm]
0  &  {\rm if}~~\sigma_1(d(\mu),\;-\gamma
e^{g(\mu)}\max_{[0,\,\mu]}F(v);\,1)=0,\\[1mm]
\sigma_1(d(\mu),\,-\gamma e^{g(\mu)}\max_{[0,\,\mu]}F(v);\,1)\;\;
& {\rm if}~~\sigma_1(d(\mu),\;-\gamma e^{g(\mu)}\max_{[0,\,\mu]}F(v);\,1)<0.
\end{array}\right.
\end{equation*}
Therefore, there exists a positive constant $\underline{\lambda}=\underline{\lambda}(\mu)$ such that
\eqref{ys12} has a positive solution $(u,v)$ only if $\lambda>\underline{\lambda}$. The proof of Theorem \ref{th11}(1) is complete.

\subsection{Proof of Theorem \ref{th11}(2)}

Assume on the contrary that \eqref{ys12} has a positive solution $(u_\lm,v_\lm)$
for large $\lambda>0$. Note that $1\leq e^{g(v_\lm)}\leq e^{g(\mu)}$. Then we
derive from (\ref{ys21}) that
\begin{equation*}
-{\rm div}\left(d(v_\lm)e^{g(v_\lm)}\nabla w\right)
=\lambda e^{g(v_\lm)}w-e^{g(v_\lm)}w\left(e^{g(v_\lm)}w-\gamma F(v_\lm)\right)
\geq \lambda w-e^{2g(\mu)}w^2,\,\,x\in\Omega.
\end{equation*}
This shows that $w$ is a supersolution of
\begin{equation}\label{ys31}
\left\{\begin{array}{ll}
-{\rm div}\left(f(v_\lm)\nabla z\right)=\lambda z-e^{2g(\mu)}z^2,\;\;
&   x\in\Omega,\\[1mm]
z=0, & x\in\partial\Omega,
\end{array}\right.
\end{equation}
where
 \[f(v_\lm)=d(v_\lm)e^{g(v_\lm)}.\]
When $\lambda>\sigma_1(f(v_\lm),\,0;\,1)$ (this is possible as
$\sigma_1(f(v_\lm),\,0;\,1)\le \sigma_1 (d(0)e^{g(\mu)},\,0;\,1)$), Lemma
\ref{le32} ensures that (\ref{ys31}) has a unique positive solution
$\theta_{f(v_\lm),\,\lambda,\,e^{2g(\mu)}}$, and
\begin{equation*}
\theta_{f(v_\lm),\,\lambda,\,e^{2g(\mu)}}(x)\geq\frac{\lambda-
\sigma_1(f(v_\lm),\,0;\,1)}{e^{2g(\mu)}
\|\phi_{\lambda}\|_{C(\overline{\Omega})}}\phi_{\lambda}(x),\;\;x\in\Omega.
\end{equation*}
Here $\phi_{\lambda}$ with $\|\phi_{\lambda}\|_{L^2(\Omega)}=1$ stands for
the positive eigenfunction associated to $\sigma_1(f(v_\lm),\,0;\,1)$.
It follows from the standard sub-supersolution arguments that
\begin{equation*}
w(x)\geq \theta_{f(v_\lm),\,\lambda,\,e^{2g(\mu)}}(x)\geq\frac{\lambda-
\sigma_1(f(v_\lm),\,0;\,1)}{e^{2g(\mu)}
\|\phi_{\lambda}\|_{C(\overline{\Omega})}}\phi_{\lambda}(x),\;\;x\in\Omega.
\end{equation*}
In view of Lemma \ref{le21} and Lemma \ref{le31},
\begin{equation*}
\sigma_1(f(v_\lm),\,0;\,1)\leq d(0)e^{g(\mu)}
\sigma_1:=s(\mu).
\end{equation*}
Let
  $$\tau(\lambda):=\frac{\lambda-s(\mu)}{e^{2g(\mu)}\|\phi_{\lambda}
 \|_{C(\overline{\Omega})}}.$$
Then
 \begin{equation*}
u_\lm(x)\geq e^{-g(v_\lm)}u_\lm(x)= w(x)\geq \tau(\lambda)\phi_{\lambda}(x),\;\;x\in\Omega.
\end{equation*}
Since $\|\phi_{\lambda}\|_{C(\overline{\Omega})}$ is
bound in $\lambda$ (see, e.g., \cite[Theorem 4.1]{sd1965}), it follows that
\begin{equation*}
\tau(\lambda)\geq\frac{\lambda-s(\mu)}{Ce^{2g(\mu)}}\rightarrow\infty
~~{\rm as}~~\lambda\rightarrow\infty.
\end{equation*}
Together with Lemma \ref{le31}, we derive from the second equation of
\eqref{ys12} that
$$\mu=\sigma_1(D,\,v_\lm+u_\lm\mathbb{F}(v_\lm);\,1)\geq
\sigma_1\left(D,\,\tau(\lambda)\phi_{\lambda}\min_{[0,\,\mu]}
\mathbb{F}(v_\lm);\,1\right):=h(\lambda).$$

Assume that $\min_{[0,\,\mu]}\mathbb{F}(v_\lm)>0$. We shall show
\begin{equation}
\lim\limits_{\lambda\rightarrow+\infty}h(\lambda)=+\infty.
 \label{3.2a}\end{equation}
Once this is done, then the conclusion (2) is deduced as $\mu>D\sigma_1$ is fixed. In the following we prove \eqref{3.2a}. Making use of Lemma \ref{le21} we have
\begin{equation*}
h(\lambda)=\inf_{\varphi\in H^1_0(\Omega)\backslash \{0\}}
\frac{\int_{\Omega}D|\nabla\varphi|^2dx+\tau(\lambda)\min_{[0,\,\mu]}
\mathbb{F}(v_\lm)\int_{\Omega}\phi_{\lambda}\varphi^2dx}{\int_{\Omega}
\varphi^2dx}.
\end{equation*}
Assume on the contrary that $h(\lambda)$ is bounded. Then we can find a sequence
$\varphi_{\lambda}\in H^1_0(\Omega)$ with
$\|\varphi_{\lambda}\|_{L^2(\Omega)}=1$ such that the infimum of $h(\lambda)$
is attained at $\varphi_{\lambda}$, that is,
\begin{equation}\label{ys32}
h(\lambda)=\int_{\Omega}D|\nabla\varphi_{\lambda}|^2dx+
\tau(\lambda)\min_{[0,\,\mu]}\mathbb{F}(v_\lm)\int_{\Omega}
\phi_{\lambda}\varphi^2_{\lambda}dx \ge\int_{\Omega}D|\nabla\varphi_{\lambda}|^2dx.
\end{equation}
The boundedness of $h(\lambda)$ shows that
$\varphi_{\lambda}$ is bounded in $H_0^1(\Omega)$. Therefore, there
exists some nonnegative function $\varphi_0\geq0~(\not \equiv 0)$ with
$\|\varphi_0\|_{L^2(\Omega)}=1$ such that
\begin{equation}\label{ys33}
\left\{\begin{array}{ll}
\varphi_{\lambda}\rightarrow\varphi_0   & {\rm weakly~in}~~
H^1_0(\Omega) \\[1mm]
\varphi_{\lambda}\rightarrow\varphi_0   & {\rm in}~~L^2(\Omega)
\end{array}\right.~~~{\rm as}~~~\lambda\rightarrow\infty,
\end{equation}
up to a subsequence if necessary. In order to take the limit on both sides
of (\ref{ys32}), we need to discuss the limit of $\phi_\lambda$ as
$\lambda\rightarrow\infty$. Since $d(\mu)\leq f(v_\lm)\leq d(0)e^{g(\mu)}$,
it follows from the monotonicity properties of the principal eigenvalue (see
Lemma \ref{le31}) that
\begin{equation*}
d(\mu)\sigma_1\leq\sigma_1(f(v_\lm),0;1)
\leq d(0)e^{g(\mu)}\sigma_1.
\end{equation*}
Thus, there exists some number
$\sigma_0\in[d(\mu)\sigma_1,d(0)e^{g(\mu)}\sigma_1]$
such that
\begin{equation*}
\sigma_1(f(v_\lm),0;1)\rightarrow\sigma_0~~
{\rm as}~~\lambda\rightarrow\infty,
\end{equation*}
up to a subsequence if necessary. Notice that $\phi_{\lambda}$ satisfies $\|\phi_{\lambda}\|_{L^2(\Omega)}=1$ and
\begin{equation}\label{ys34}
\left\{\begin{array}{ll}
-{\rm div}\left(f(v_\lm)\nabla\phi_{\lambda}
\right)=\sigma_1(f(v_\lm),0;1)\phi_{\lambda},\;\; & x \in \Omega,\\[1mm]
\phi_{\lambda}=0, & x \in \partial\Omega.
\end{array}\right.
\end{equation}
Our multiplying both sides of (\ref{ys34}) by $\phi_\lambda$ and integrating
over $\Omega$ gives
\begin{equation*}
d(\mu)\int_{\Omega}|\nabla \phi_{\lambda}|^2dx
\leq \int_{\Omega}f(v_\lm)|\nabla \phi_{\lambda}|^2dx
=\sigma_1(f(v_\lm),0;1)\int_{\Omega}\phi_{\lambda}^2dx
\leq d(0)e^{g(\mu)}\sigma_1.
\end{equation*}
Thus, $\phi_{\lambda}$ is bounded in $H_0^1(\Omega)$, and so there
exists some nonnegative function $\phi_0\geq0~(\not \equiv 0)$ with
$\|\phi_0\|_{L^2(\Omega)}=1$ such that
\begin{equation}\label{ys35}
\left\{\begin{array}{ll}
\phi_{\lambda}\rightarrow \phi_0\; \; & {\rm weakly~in}~~
H^1_0(\Omega) \\[1mm]
\phi_{\lambda}\rightarrow\phi_0 \; & {\rm in}~~L^2(\Omega)
\end{array}\right.~~~{\rm as}~~~\lambda\rightarrow\infty.
\end{equation}
Noticing that $a_\lambda(x,\xi)=f(v_\lm(x))\xi$ satisfies all assumptions of
\cite[Theorem 2.2]{bcr2006} (see \cite[Example 1.1]{bcr2006}). It follows that there
exists a uniformly elliptic symmetric matrix $A_{\rm hom}\in(L^{\infty}(\Omega))^{n\times n}$
such that the operators $-{\rm div}(f(v_\lm)\nabla u)$ $G$-converge to (up to a subsequence) the operator
$-{\rm div}(A_{\rm hom}\nabla u)$ as $\lm\to\infty$ (see, e.g., \cite[Remark 2.5]{bcr2006}
or \cite[Theorem 4.1]{CMD90}).
Consequently, we can apply the homogenization technique (see,
e.g.,  \cite[Theorem 2.2]{bcr2006} or \cite[Theorem 2.1]{k1979}) to conclude
that $(A_{\rm hom}, \sigma_0, \phi_0)$ satisfies
\begin{equation*}
-{\rm div}(A_{\rm hom}\nabla\phi_0)=\sigma_0\phi_0.
\end{equation*}
According to the strong maximum principle, we have
$\phi_0>0$ in $\Omega$ since $\sigma_0\phi_0\geq 0~(\not\equiv 0)$.
Our setting $\lambda\rightarrow\infty$ in (\ref{ys32}) yields
$$\limsup_{\lambda\rightarrow\infty} \int_{\Omega}\phi_{\lambda}
\varphi_{\lambda}^2dx=0$$
since $\tau(\lambda)\rightarrow\infty$ as $\lambda\rightarrow\infty$ and
$\min_{[0,\,\mu]}\mathbb{F}(v_\lm)>0$. On the other hand, it follows from (\ref{ys33})
and (\ref{ys35}) that
\begin{equation*}
\limsup\limits_{\lambda\rightarrow\infty}
\int_{\Omega}\phi_{\lambda}\varphi_{\lambda}^2dx
=\int_{\Omega}\phi_0\varphi_0^2dx>0,
\end{equation*}
a contradiction. The proof of Theorem \ref{th11}(2) is complete.

\subsection{Proof of Theorem \ref{th11}(3)}
%The purpose of this subsection is to investigate the bifurcation structure
%of the set of positive solutions of \eqref{ys12} in the case that $F'(0)=0$
%and complete the proof of Theorem \ref{th11}(3).
%Ecologically, Theorem \ref{th11}(3) shows that the prey can coexist with
%the predator even though its growth rate is too high, provided that the
%interaction between two species satisfies $\min_{[0,\,\mu]}\mathbb{F}(v)>0$.

As the functional framework in the bifurcation analysis for system \eqref{ys12},
we employ the Banach spaces
$$X:=W^{2,p}(\Omega)\cap W^{1,p}_0(\Omega)~~{\rm and}~~Y:=L^p(\Omega)~~
{\rm with}~~p>n.$$
Define the associated operator
${\mathscr L}:\mathbb{R}\times X\times X\mapsto Y\times Y$ by
%\begin{equation*}
%{\mathscr L}:\mathbb{R}\times X\times X\mapsto Y\times Y
%\end{equation*}
%given by
\begin{equation*}
{\mathscr L}{(\lambda,u,v)}=\left(\begin{matrix}
-{\rm div}\left(d(v)\nabla u-u\chi(v)\nabla v\right)-\lambda u+u^2
-\gamma u F(v)  \\[1mm]
-D \Delta v-\mu v+v^2+u F(v)
\end{matrix}\right).
\end{equation*}
This means that $(u,v)\in X\times X$ is a nonnegative solution of system
\eqref{ys12} if and only if ${\mathscr L}(\lambda,u,v)=0$.

Taking $\lambda$ as the bifurcation parameter, we apply the well-known local
bifurcation theorem of Crandall and Rabinowitz from a simple eigenvalue (see
\cite[Theorem 1.7]{cr1971}) to find a bifurcation point on
$\Gamma_{v}=\{(\lambda,0,\omega_\mu): \lambda\in \mathbb{R}\}$ with $\mu>D \sigma_1$
where positive solutions of \eqref{ys12} emanate. More precisely, the following local bifurcation properties hold.

\begin{Lemma}\label{le33}\, Let $\mu>D\sigma_1$, and let $\lambda_\mu$ and $(\phi_\mu, \psi_\mu)$ be given by \eqref{1.8a} and \eqref{1.10}, respectively. Then positive solutions of \eqref{ys12} emanate from $\Gamma_{v}$ if and only if
$\lambda=\lambda_\mu$. To be precise, there is a small neighborhood
$\mathcal{N}_1$ of $(\lambda,u,v)=(\lambda_\mu,0,\omega_\mu)$
in $\mathbb{R}\times X\times X$ such that
${\mathscr L}^{-1}(0)\cap \mathcal{N}_1$ consists of the union of
$\Gamma_{v}\cap \mathcal{N}_1$ and the local curve
\begin{equation}
(\lambda(s),u(s),v(s))=(\lambda_\mu+\lambda_\mu(s),
s(\phi_\mu+\phi(s)),\omega_\mu+
s(\psi_\mu+\psi(s)))
  \label{3.7}\end{equation}
for $s\in(-\varepsilon,\varepsilon)$ with some $\varepsilon>0$, where
$(\lambda_\mu(s),\phi(s),\psi(s))\in \mathbb{R}\times X\times X$ is
continuously differentiable for $s\in(-\varepsilon,\varepsilon)$ and
satisfies $(\lambda_\mu(0),\phi(0),\psi(0))=(0,0,0)$, and
$\lambda_\mu^{\prime}(0)$ is given by \eqref{ys314}.
Therefore, positive solutions contained in
${\mathscr L}^{-1}(0)\cap \mathcal{N}_1$ can be expressed as
\begin{equation}
\mathcal{S}_1=\{(\lambda_\mu+\lambda_\mu(s),
s(\phi_\mu+\phi(s)),\omega_\mu+
s(\psi_\mu+\psi(s))): s\in(0,\varepsilon)\}.
\label{3.8}\end{equation}
\end{Lemma}

\begin{proof} For clarity, we divide the proof into the following four steps.

\textbf{Step 1:} We claim that the kernel $\mathrm{Ker}[{\mathscr L}_{(u,v)}
(\lambda_\mu,0,\omega_\mu)]$ is one-dimensional. By straightforward
calculations, the linearized operator of ${\mathscr L}(\lambda_\mu,u,v)$ around
the semitrivial solution $(0,\omega_\mu)$ is given by
\begin{equation*}
{\mathscr L}_{(u,v)}(\lambda_\mu,0,\omega_\mu)
=\left(\begin{matrix}
-{\rm div}\left(d(\omega_\mu)\nabla \cdot
-\chi(\omega_\mu)\nabla\omega_\mu\cdot\right)
-\lambda_\mu-\gamma F(\omega_\mu) & 0\\[1mm]
F(\omega_\mu) & -D\Delta-\mu+2\omega_\mu
\end{matrix}\right).
 \end{equation*}
Let $(\phi,\psi)\in\mathrm{Ker}[{\mathscr L}_{(u,v)}(\lambda_\mu,0, \omega_\mu)]$. Then
\begin{equation*}
\left\{\begin{array}{ll}
-{\rm div}\left(d(\omega_\mu)\nabla\phi-\chi(\omega_\mu)\nabla
\omega_\mu\phi\right)-\gamma F(\omega_\mu)\phi=\lambda_\mu\phi,
& x\in\Omega,\\[1mm]
-D\Delta\psi+(2\omega_\mu-\mu)\psi=-F(\omega_\mu)\phi,
& x\in\Omega,\\[1mm]
\phi=\psi=0, & x\in\partial\Omega.
\end{array}\right.
\end{equation*}
By setting $\Phi=e^{-g(\omega_\mu)}\phi$, where
$g(\omega_\mu)=\int_0^{\omega_\mu}\frac{\chi(\tau)}{d(\tau)}d\tau$,
we can reduce the above problem to
\begin{equation}\label{ys37}
\left\{\begin{array}{ll}
-{\rm div}\left(d(\omega_\mu)e^{g(\omega_\mu)}\nabla\Phi\right)
-\gamma F(\omega_\mu)e^{g(\omega_\mu)}\Phi=\lambda_\mu
e^{g(\omega_\mu)}\Phi, & x\in\Omega,\\[1mm]
-D\Delta\psi+(2\omega_\mu-\mu)\psi=-F(\omega_\mu)
e^{g(\omega_\mu)}\Phi, & x\in\Omega,\\[1mm]
\Phi=\psi=0, & x\in\partial\Omega.
\end{array}\right.
\end{equation}
From the equations of $\Phi$ we have $\Phi=c\Phi_\mu$ for some constant $c$ by \eqref{1.9}. Then $\phi=c\phi_\mu$ and $\psi=c\psi_\mu$ by \eqref{1.10}. Therefore, \begin{equation*}
\mathrm{Ker}[{\mathscr L}_{(u,v)}(\lambda_\mu,0,
\omega_\mu)]=\mathrm{span}\langle
(\phi_\mu,\psi_\mu)\rangle,
\end{equation*}
and thus $\mathrm{Ker}[{\mathscr L}_{(u,v)}(\lambda_\mu,0,\omega_\mu)]$
is one-dimensional.

\textbf{Step 2:} We claim that the codimension of $\mathrm{Range}[{\mathscr L}_{(u,v)}(\lambda_\mu,0,\omega_\mu)]$ is $1$. In fact, if
$$(w,z)\in \mathrm{Range}[{\mathscr L}_{(u,v)}(\lambda_\mu,0,
\omega_\mu)],$$
then there exists $(\phi,\psi)\in X\times X$ such that
\begin{equation}\label{ys38}
\left\{\begin{array}{ll}
-{\rm div}\left(d(\omega_\mu)\nabla\phi-\chi(\omega_\mu)\nabla
\omega_\mu\phi\right)-(\lambda_\mu+\gamma F(\omega_\mu))\phi=w,
& x\in\Omega, \\[1mm]
-D\Delta\psi+(2\omega_\mu-\mu)\psi+F(\omega_\mu)\phi=z,
& x\in\Omega, \\[1mm]
\phi=\psi=0,  &  x\in\partial\Omega.
\end{array}\right.
\end{equation}
Let $\Phi=e^{-g(\omega_\mu)}\phi$. Then $\Phi$ satisfies
\begin{equation}\label{ys38'}
\left\{\begin{array}{ll}
-{\rm div}\left(d(\omega_\mu)e^{g(\omega_\mu)}\nabla\Phi\right)
-(\lambda_\mu+\gamma F(\omega_\mu))e^{g(\omega_\mu)}\Phi=w,
&  x\in \Omega,  \\[1mm]
\Phi=0,  &  x\in\partial\Omega.
\end{array}\right.
\end{equation}
Since the operator
\begin{equation*}
-{\rm div}\left(d(\omega_\mu)e^{g(\omega_\mu)}\nabla\right)
-(\lambda_\mu+\gamma F(\omega_\mu))e^{g(\omega_\mu)}:X\mapsto Y
\end{equation*}
is self-adjoint, it follows from the Fredholm alternative theorem that
(\ref{ys38'}) has a solution if and only if
$$\int_{\Omega}w\Phi_\mu dx=0,$$
where $\Phi_\mu$ is the positive eigenfunction associated
to $\lambda_\mu$. In this situation, the second equation of (\ref{ys38}) admits the solution
\begin{equation*}
\psi=(-D\Delta+2\omega_\mu-\mu)^{-1}\left(z-F(\omega_\mu)
e^{g(\omega_\mu)}\Phi\right)
\end{equation*}
since the operator $-D\Delta+2\omega_\mu-\mu: X\mapsto Y$ is invertible.
This shows that
$$\mathrm{Range}[{\mathscr L}_{(u,v)}(\lambda_\mu,0,\omega_\mu)]=
\{\mathrm{span}\langle(\Phi_\mu,0)\rangle\}^{\perp},$$
and thus the codimension of $\mathrm{Range}[{\mathscr L}_{(u,v)}(\lambda_\mu,0,
\omega_\mu)]$ is $1$.

\textbf{Step 3:} We shall show that
\begin{equation}\label{ys39}
{\mathscr L}_{(u,v),\lambda}(\lambda_\mu,0,
\omega_\mu)\left(\begin{matrix}
\phi_\mu\\
\psi_\mu\end{matrix}\right)\not\in
\mathrm{Range}[{\mathscr L}_{(u,v)}(\lambda_\mu,
0,\omega_\mu)].
\end{equation}
The straightforward calculations yield
\begin{equation*}
{\mathscr L}_{(u,v),\lambda}(\lambda_\mu,0,\omega_\mu)
\left(\begin{matrix}\phi_\mu\\
\psi_\mu\end{matrix}\right)=\left(\begin{matrix}
-\phi_\mu\\
0 \end{matrix}\right).
\end{equation*}
If (\ref{ys39}) is false, then we can find a pair of functions
$(\phi,\psi)\in X\times X$ such that
\begin{equation*}
\left\{\begin{array}{ll}
-{\rm div}\left(d(\omega_\mu)\nabla\phi-\chi(\omega_\mu)\nabla
\omega_\mu\phi\right)-(\lambda_\mu+\gamma F(\omega_\mu))\phi=
-\phi_\mu,  &   x\in\Omega,  \\[1mm]
-D\Delta\psi+(2\omega_\mu-\mu)\psi+F(\omega_\mu)\phi=0,
&  x\in\Omega,  \\[1mm]
\phi=\psi=0,  &  x\in\partial\Omega.
\end{array}\right.
\end{equation*}
As above, by the Fredholm alternative theorem, we get
$$0=\int_{\Omega}\phi_\mu\Phi_\mu dx=
\int_{\Omega}e^{g(\omega_\mu)}\Phi^2_{\lambda_\mu}dx.$$
This is impossible since $\Phi_\mu>0$ in $\Omega$. Consequently,
all the assumptions for applying the local bifurcation theorem of Crandall
and Rabinowitz from a simple eigenvalue have been checked, and thus, the
local bifurcation curve $(\lambda(s),u(s),v(s))$ with the form \eqref{3.7} for $s\in(-\varepsilon,\varepsilon)$ is obtained.

\textbf{Step 4:} We calculate $\lambda_\mu^{\prime}(0)$.
By setting $w(s)=e^{-g(v(s))}u(s)$, we obtain from (\ref{ys21}) that $w(s)$
satisfies
\begin{equation*}
\left\{\begin{array}{ll}
-{\rm div}\left(d(v(s))e^{g(v(s))}\nabla w(s)\right)=e^{g(v(s))}w(s)
(\lambda(s)-u(s)+\gamma F(v(s))),\;\; & x\in\Omega,\\[1mm]
w(s)=0, & x\in\partial\Omega.
\end{array}\right.
\end{equation*}
Multiplying by $\Phi_\mu$ and integrating the result we have
\begin{equation*}
\int_{\Omega}d(v(s))e^{g(v(s))}\nabla w(s)\cdot\nabla \Phi_\mu dx
=\int_{\Omega}e^{g(v(s))}w(s)(\lambda(s)-u(s)+\gamma F(v(s)))
\Phi_\mu dx.
\end{equation*}
Denote
\begin{equation*}
\left\{\begin{aligned}
&A(s):=e^{g(v(s))}=e^{g(\omega_\mu)}+\widetilde{A}(s), \\
&B(s):=\gamma F(v(s))=\gamma F(\omega_\mu)+\widetilde{B}(s),\\
&C(s):=d(v(s))e^{g(v(s))}=d(\omega_\mu)e^{g(\omega_\mu)}
+\widetilde{C}(s).
\end{aligned}\right.
\end{equation*}
Then
\begin{equation*}
\lim\limits_{s\rightarrow 0}\frac{\widetilde{A}(s)}{s}=A^{\prime}(0),
~~~\lim\limits_{s\rightarrow 0}\frac{\widetilde{B}(s)}{s}=B^{\prime}(0),
~~~\lim\limits_{s\rightarrow 0}\frac{\widetilde{C}(s)}{s}=C^{\prime}(0).
\end{equation*}
After some rearrangement, we obtain
\bes\label{ys311}
&&\int_{\Omega}\left(d(\omega_\mu)e^{g(\omega_\mu)}
+\widetilde{C}(s)\right)\nabla w(s)\cdot\nabla\Phi_\mu dx\nonumber\\
&=&\int_{\Omega}\left(e^{g(\omega_\mu)}+\widetilde{A}(s)\right)
w(s)(\lambda_\mu+\lambda_\mu(s))\Phi_\mu dx\nonumber\\
&&+\int_{\Omega}\left(e^{g(\omega_\mu)}+\widetilde{A}(s)\right)
w(s)\left(-u(s)+\gamma F(\omega_\mu)+\widetilde{B}(s)\right)
\Phi_\mu dx.
\ees
Multiplying (\ref{1.9}) by $w(s)$ and integrating the result we obtain
\begin{equation}\label{ys312}
\int_{\Omega} d(\omega_\mu)e^{g(\omega_\mu)}
\nabla\Phi_\mu\cdot\nabla w(s)dx=
\int_{\Omega}\left(\gamma F(\omega_\mu)+\lambda_\mu\right)
e^{g(\omega_\mu)}\Phi_\mu w(s)dx.
\end{equation}
It follows from (\ref{ys311}) and (\ref{ys312}) that
\begin{equation*}
{\begin{aligned}
   \int_{\Omega}\widetilde{C}(s)\nabla w(s)\cdot\nabla\Phi_\mu dx
=& \int_{\Omega}e^{g(\omega_\mu)}w(s)\left(\lambda_\mu(s)-u(s)+
   \widetilde{B}(s)\right)\Phi_\mu dx  \\
 & ~~+\int_{\Omega}\widetilde{A}(s)w(s)\left(\lambda_\mu+\lambda_\mu(s)
   -u(s)+\gamma F(\omega_\mu)+\widetilde{B}(s)\right)
   \Phi_\mu dx.\end{aligned}}
\end{equation*}
Dividing the above equation by $s^2$ firstly,
and letting $s\rightarrow 0^{+}$ and using \eqref{3.7} secondly, we have
\begin{equation}\label{ys313}
{\begin{aligned}
   \int_{\Omega}C^{\prime}(0)\left|\nabla\Phi_\mu\right|^2dx
=& \int_{\Omega}e^{g(\omega_\mu)}\left(\lambda_\mu^{\prime}(0)
   -e^{g(\omega_\mu)}\Phi_\mu+B^{\prime}(0)\right)
   \Phi_\mu^2dx  \\
 & ~~+\int_{\Omega}A^{\prime}(0)\left(\lambda_\mu+\gamma F(\omega_\mu)
   \right)\Phi_\mu^2dx.
\end{aligned}}
 \end{equation}
Making use of the expression $g(v)=\int_0^v\frac{\chi(\tau)}{d(\tau)}d\tau$ and \eqref{3.7}, we can deduce by straightforward calculations that
\begin{equation}\label{ys313'}
\left\{\begin{aligned}
&A^{\prime}(0)=\frac{\chi(\omega_\mu)}{d(\omega_\mu)}
e^{g(\omega_\mu)}\psi_\mu,\\
&B^{\prime}(0)=\gamma F^{\prime}(\omega_\mu)\psi_\mu,\\
&C^{\prime}(0)=\left(d^{\prime}(\omega_\mu)+\chi(\omega_\mu)
\right)e^{g(\omega_\mu)}\psi_\mu.
\end{aligned}\right.
\end{equation}
Substitute (\ref{ys313'}) into (\ref{ys313}) to derive
(\ref{ys314}). The proof of Lemma \ref{le33} is complete.
\end{proof}

\begin{Remark}
Lemma \ref{le33} still holds in the case that $F'(0)>0$.
\end{Remark}

In the following we investigate the stabilities of $(\theta_\lambda,0)$ and $(0,\omega_\mu)$. By direct calculations, the linearized operator of
${\mathscr L}(\lambda,u,v)$ around the semitrivial solution
$(\theta_\lambda,0)$ reads as
\begin{equation*}
{\mathscr L}_{(u,v)}(\lambda,\theta_\lambda,0)=\left(\begin{matrix}
-d(0)\Delta-\lambda+2\theta_\lambda\;\; &\;\;
-{\rm div}\left(d^{\prime}(0)\nabla\theta_\lambda\cdot-
\theta_\lambda\chi(0)\nabla\cdot\right)-\gamma\theta_\lambda
F'(0) \\[1mm]
0 & -D\Delta-\mu+\theta_\lambda F'(0)
\end{matrix}\right).
\end{equation*}
By the well-known Riesz-Schauder theory (see, e.g., \cite[Lemma 3.5]{h1984}), the spectrum
$\rho({\mathscr L}_{(u,v)}(\lambda,\theta_\lambda,0))$ of
${\mathscr L}_{(u,v)}(\lambda,\theta_\lambda,0)$ only consists of real eigenvalues and
\begin{equation*}
\rho({\mathscr L}_{(u,v)}(\lambda,\theta_\lambda,0))
=\rho(-d(0)\Delta-\lambda+2\theta_\lambda)\cup
\rho(-D\Delta-\mu+\theta_\lambda F'(0)).
\end{equation*}
As $\sigma_1(d(0),\,-\lambda+2\theta_{d(0),
\lambda};1)>0$, all eigenvalues of $-d(0)\Delta-\lambda+2\theta_\lambda$ are positive.
Taking advantage of the linearization principle for quasilinear evolution equations
\cite{p1981} we have the following conclusions:

(i)\; If $F'(0)=0$, then the first eigenvalue of $-D\Delta-\mu$ is $D\sigma_1-\mu$, which is negative since $\mu>D\sigma_1$. Thus, $(\theta_\lambda,0)$ is unstable.

(ii)\, If $F'(0)>0$, then the first eigenvalue of $-D\Delta-\mu+\theta_\lambda F'(0)$ is $\mu_\lambda-\mu$, where $\mu_\lambda=\sigma_1(D,\,\theta_\lambda F'(0);\,1)$. Therefore,
$(\theta_\lambda,0)$ is unstable when $\mu>\mu_\lambda$ and it is
asymptotically stable when $\mu<\mu_\lambda$.
According to Lemma \ref{le31} and Lemma \ref{le32}, $\mu_{\lambda}$ is a continuous and increasing function with respect to $\lambda$. By virtue of \cite[Proposition 1.2]{y2008}, we see that $\lim_{\lambda\rightarrow d(0)\sigma_1} \theta_\lambda=0$ uniformly in $\overline{\Omega}$ and
$\lim_{\lambda\rightarrow\infty}\theta_\lambda(x)=\infty$ for each
$x\in\Omega$. Together with \cite[Proposition 1.1]{y2008}, we can obtain
that $\lim_{\lambda\rightarrow d(0)\sigma_1}\mu_{\lambda}=
D\sigma_1$ and $\lim_{\lambda\rightarrow\infty}\mu_{\lambda}
=\infty$. Owing to implicit function theorem, there exists a unique number
$\lambda^*=\lambda^*(\mu)\in(d(0)\sigma_1,\infty)$ such
that $\mu=\mu_{\lambda^*}$ for any given $\mu>D\sigma_1$.

Therefore, the stability of $(\theta_\lambda,0)$ is stated as follows.

\begin{Lemma}\label{le34}
Assume that $\mu>D\sigma_1$ is fixed. If $F'(0)=0$, then the
semitrivial solution $(\theta_\lambda,0)$ is unstable for
$\lambda>d(0)\sigma_1$. If $F'(0)>0$, then the semitrivial
solution $(\theta_\lambda,0)$ is asymptotically stable for
$\lambda>\lambda^*$, while it is unstable for $\lambda<\lambda^*$.
\end{Lemma}

Similarly, the stability of $(0,\omega_\mu)$ is stated as follows.

\begin{Lemma}\label{le34'}
Assume that $\mu>D\sigma_1$ is fixed. Then the semitrivial
solution $(0,\omega_\mu)$ is asymptotically stable for
$\lambda<\lambda_\mu$, while it is unstable for $\lambda>\lambda_\mu$.
\end{Lemma}

When $F'(0)=0$, the following lemma gives the global structure of the local
curve $\mathcal{S}_1$ given by (\ref{3.8}) in the $(\lambda,u,v)$ plane.

\begin{Lemma}\label{le35}
Assume that $\mu>D\sigma_1$ and $F'(0)=0$. Then the local curve $\mathcal{S}_1$
 given by \eqref{3.8} can be extended to an unbounded global continuum of positive
solutions to \eqref{ys12} along the positive values of $\lambda$.
\end{Lemma}

\begin{proof} (1)\, Let $\mathcal{P}$ be the positive
cone in $C_0^1(\overline\Omega)$ whose interior
${\rm int}\left(\mathcal{P}\right)$ is nonempty, and let $\mathcal{S}_1$ be given by \eqref{3.8}. In view of the unilateral global bifurcation for quasilinear elliptic
systems derived in \cite{cms2019}, which is based on the unilateral
bifurcation results in \cite{l2001}, there exists a continuum
\begin{equation*}
\mathcal{C}\subset \mathbb{R}\times {\rm int}\left(\mathcal{P}\right)\times
{\rm int}\left(\mathcal{P}\right)
\end{equation*}
of positive solutions of \eqref{ys12}  such that
$\mathcal{S}_1\subset \mathcal{C}$ and $\mathcal{C}$ satisfies one of the
following:

\begin{itemize}[itemsep= -2 pt,topsep = 0 pt, leftmargin = 35 pt]
  \item [(a)] $\mathcal{C}$ is unbounded in $\mathbb{R}\times
  C_0^1(\overline{\Omega})\times C_0^1(\overline{\Omega})$;

\item [(b)] there exists a positive solution $\theta_{\lambda^*}$
  of (\ref{1.6}) such that $\left(\lambda^*,\theta_{\lambda^*},0
  \right)\in \overline{\mathcal{C}}$, where $\lambda^*$ is determined by
  \begin{equation*}
  \mu=\sigma_1(D,\,\theta_{\lambda^*}F^{\prime}(0);\,1);
  \end{equation*}

  \item [(c)] there exists another positive solution $v$ of
  (\ref{ys16}) and $v\neq\omega_\mu$ such that
  \begin{equation*}
  \left(\sigma_1(d(v)e^{g(v)},\,-\gamma F(v)
  e^{g(v)};\,e^{g(v)}),\;0,\;v\right)\in \overline{\mathcal{C}};
  \end{equation*}

  \item [(d)] $\mu=D\sigma_1$ and
  $(D\sigma_1,0,0)\in \overline{\mathcal{C}}$.
\end{itemize}

We claim that the items (b), (c) and (d) cannot occur. Since
$\mu>D\sigma_1$ is fixed, the item (d) cannot be true.
As a result of the uniqueness of positive solution of (\ref{ys16}), the
item (c) cannot occur as well. If the item (b) occurs, then
$(\lambda^*,\theta_{\lambda^*},0)\in \mathbb{R}\times X\times X$
is a bifurcation point in the sense of being the limit of a sequence in
$\mathcal{C}$. That is, there exists a sequence $\{(\lambda_i,u_i,v_i)\}\subset
\mathcal{C}$ such that
 $$(\lambda_i,u_i,v_i)\rightarrow(\lambda^*,\theta_{\lambda^*},0)~~~
{\rm in}~~~\mathbb{R}\times C_0^1(\overline{\Omega})\times C_0^1(\overline{\Omega}).$$
Noticing that $(\lambda_i,u_i,v_i)$ satisfies
\begin{equation*}
\left\{\begin{array}{ll}
-{\rm div}\left(d(v_i)\nabla u_i-u_i\chi(v_i)\nabla v_i\right)=
\lambda_i u_i-u_i^2+\gamma u_i F(v_i),  & x \in \Omega,\\[1mm]
-D \Delta v_i=\mu v_i-v_i^2-u_i F(v_i),  & x \in \Omega,\\[1mm]
u_i=v_i=0,  & x \in \partial \Omega.
\end{array}\right.
\end{equation*}
Define $\hat{v}_i={v_i}/{\|v_i\|_{C(\overline{\Omega})}}$. Then
$\hat{v}_i$ satisfies
\bess
\left\{\begin{array}{ll}
-D\Delta \hat{v}_i=\hat{v}_i\left(\mu -v_i-u_i \mathbb{F}(v_i)\right),
~~&x\in\Omega,\\[1mm]
\hat{v}_i=0,~~ &x\in\partial\Omega.
\end{array}\right.
\eess
By virtue of Lemma \ref{le21} and the $L^p$ theory for elliptic equations, it follows that
$\hat{v}_i$ is bounded in $W^{2,p}(\Omega)$ for all $p>1$. Thus,
by choosing a subsequence if necessary, we may assume
$$\hat{v}_i\rightarrow \hat{v}_\infty~~~{\rm in}~~~
C_0^1(\overline{\Omega})~~~{\rm with}~~~
\|\hat{v}_\infty\|_{C(\overline{\Omega})}=1.$$
Observe that $F'(0)=\mathbb{F}(0)=0$ and $v_i\to 0$ as $i\to\infty$. Letting  $i\rightarrow\infty$ we see that $\hat{v}_\infty$ satisfies
 \bess
\left\{\begin{array}{ll}
-D\Delta \hat{v}_\infty=\mu\hat{v}_\infty,~~&x\in\Omega,\\[1mm]
\hat{v}_\infty=0,~~ &x\in\partial\Omega
 \end{array}\right.
 \eess
in the weak sense. Since $\hat{v}_\infty\geq0$ and
$\|\hat{v}_\infty\|_{C(\overline{\Omega})}=1$, we must have
$\mu=D\sigma_1$, which is a contradiction. Hence, the item (b)
cannot be satisfied.

Consequently, the only possibility is the item (a), and so $\mathcal{C}$ is
unbounded in $\mathbb{R}\times C_0^1(\overline\Omega)\times C_0^1(
\overline\Omega)$. From Lemma \ref{le22} and the Sobolev embedding theorem, it
follows that any positive solution of \eqref{ys12} is bounded in
$C_0^1(\overline\Omega)\times C_0^1(\overline\Omega)$, provided
$\lambda$ is bounded. As a consequence, $\mathrm{Proj}_{\lambda}\mathcal{C}$
is unbounded. Additionally, Theorem \ref{th11}(1) shows that
$\mathrm{Proj}_{\lambda}\mathcal{C}$ is bounded below, and hence,
$\mathrm{Proj}_{\lambda}\mathcal{C}\supset(\lambda_\mu,\infty)$. The proof
of Lemma \ref{le35} is complete.
\end{proof}

Theorem \ref{th11}(3) will follow as a consequence of the results in this
subsection.

\subsection{Proof of Theorem \ref{th11}(4)}
%The purpose of this subsection is to investigate the bifurcation structure
%of the set of positive solutions of \eqref{ys12} in the case that $F'(0)>0$
%and complete the proof of Theorem \ref{th11}(4).
%Ecologically, Theorem \ref{th11}(4) shows that the prey can coexist with
%the predator only if its growth rate is not too low and too high, provided
%that the interaction between two species satisfies
%$\min_{[0,\,\mu]}\mathbb{F}(v)>0$.

As shown in Lemma \ref{le34}, the stability of $(\theta_\lambda,0)$
will change as $\lambda$ increases across $\lambda^*$. Moreover, we will
prove that $(\lambda^*,\theta_{\lambda^*},0)$ is a bifurcation point
on $\Gamma_u$ from which positive solutions of \eqref{ys12} emanate. Our
argument below is very similar to that of Lemma \ref{le33}, and hence we will
only sketch it here.

Let $\psi_{\lambda^*}$ be the positive eigenfunction associated to $\sigma_1(D,\,\theta_{\lambda^*}F'(0);\,1)$, and set
 \begin{equation*}
\phi_{\lambda^*}=(-d(0)\Delta+2\theta_{\lambda^*}-\lambda^*)^{-1}
\left[{\rm div}\left(d^{\prime}(0)\nabla\theta_{\lambda^*}
\psi_{\lambda^*}-\theta_{\lambda^*}\chi(0)\nabla\psi_{\lambda^*}
\right)+\gamma\theta_{\lambda^*}F'(0)\psi_{\lambda^*}\right].
\end{equation*}

\begin{Lemma}\label{le36}
Assume that $\mu>D\sigma_1$ and $F'(0)>0$. Then positive
solutions of \eqref{ys12} emanate from $\Gamma_{u}$ if and only if
$\lambda=\lambda^*$. To be precise, there is a small neighborhood
$\mathcal{N}_2$ of $(\lambda,u,v)=(\lambda^*,\theta_{d(0),
\lambda^*},0)$ in $\mathbb{R}\times X\times X$ such that
${\mathscr L}^{-1}(0)\cap \mathcal{N}_2$ consists of the union of
$\Gamma_{u}\cap \mathcal{N}_2$ and the local curve
\begin{equation*}
(\lambda(s),u(s),v(s))=(\lambda^*+\lambda^*(s),\theta_{\lambda^*}
+s(\phi_{\lambda^*}+\phi(s)),s(\psi_{\lambda^*}+\psi(s)))
\end{equation*}
for $s\in(-\varepsilon,\varepsilon)$ with some $\varepsilon>0$, where
$(\lambda^*(s),\phi(s),\psi(s))\in \mathbb{R}\times X\times X$ is
continuously differentiable for $s\in(-\varepsilon,\varepsilon)$ and
satisfies $(\lambda^*(0),\phi(0),\psi(0))=(0,0,0)$. Therefore, positive
solutions contained in ${\mathscr L}^{-1}(0)\cap \mathcal{N}_2$ can be
expressed as
\begin{equation*}
\mathcal{S}_2=\{(\lambda^*+\lambda^*(s),\theta_{\lambda^*}
+s(\phi_{\lambda^*}+\phi(s)),s(\psi_{\lambda^*}+\psi(s))):
s\in(0,\varepsilon)\}.
\end{equation*}
\end{Lemma}

\begin{proof}
We employ the change of variables $w=u-\theta_\lambda$ and define
the operator
${\mathscr K}:\mathbb{R}\times X\times X\mapsto Y\times Y$ by
$${\mathscr K}(\lambda,w,v):={\mathscr L}{\left(\lambda,
w+\theta_\lambda,v\right)}.$$
Notice that the change of variables maps $\Gamma_{u}$ to the set
$\{(\lambda,0,0):\lambda>d(0)\sigma_1\}$. We next seek for
a bifurcation point on this set from which solutions with $v>0$ of
${\mathscr K}{(\lambda,w,v)}=0$ emanate. By straightforward
calculations, the linearized operator of ${\mathscr K}_{(w,v)}(\lambda,w,v)$
around $(0,0)$ is given by
$${\mathscr K}_{(w,v)}(\lambda,0,0)=\left(\begin{array}{cc}
   -d(0)\Delta-\lambda+2\theta_\lambda\;\;
& \;\; -{\rm div}\left(d^{\prime}(0)\nabla\theta_\lambda\cdot
   -\theta_\lambda\chi(0)\nabla\cdot\right)-\gamma
   \theta_\lambda F^{\prime}(0)  \\ [2mm]
   0&  -D\Delta-\mu+\theta_\lambda F^{\prime}(0)
 \end{array}\right).$$
Now we determine $\mathrm{Ker}[{\mathscr K}_{(w,v)}(\lambda^*,0,0)]$.
When $(\phi,\psi)\in\mathrm{Ker}[{\mathscr K}_{(w,v)}(\lambda^*,0,0)]$, there holds:
$$\left\{\begin{array}{ll}
-d(0)\Delta\phi+(2\theta_{\lambda^*}-\lambda^*)\phi-
{\rm div}\left(d^{\prime}(0)\nabla\theta_{\lambda^*}\psi
-\theta_{\lambda^*}\chi(0)\nabla\psi\right)=\gamma\theta_{\lambda^*}
F^{\prime}(0)\psi,\;\; & x\in\Omega,  \\[1mm]
-D\Delta\psi+\theta_{\lambda^*} F^{\prime}(0)\psi=\mu\psi,
&  x\in\Omega,  \\[1mm]
\phi=\psi=0,  &  x\in\partial\Omega.
 \end{array}\right.$$
By the definition of $\lm^*=\lm^*(\mu)$, we have $\mu=\mu_{\lm^*}=\sigma_1(D,\,\theta_{\lambda^*}F'(0);\,1)$.
It follows that  $\psi=c\psi_{\lambda^*}$ for some constant $c$. Consequently, $\phi=c\phi_{\lambda^*}$.
This shows that $\mathrm{Ker}[{\mathscr K}_{(w,v)}(\lambda^*,0,0)]=
\mathrm{span}\langle(\phi_{\lambda^*},\psi_{\lambda^*})\rangle$, and
thereby, $\mathrm{Ker}[{\mathscr K}_{(w,v)}(\lambda^*,0,0)]$
is one-dimensional.

As shown in the proof of Step 2 of Lemma \ref{le33}, we can prove that
$(w,z)\in \mathrm{Range}[{\mathscr K}_{(w,v)}(\lambda^*,0,0)]$
implies $\int_{\Omega}z\psi_{\lambda^*}dx=0$. Thus, the codimension of
$\mathrm{Range}[{\mathscr K}_{(U,v)}(\lambda^*,0,0)]$ is $1$. We next verify the transversality condition
 \bes{\mathscr K}_{(w,v),\lambda}(\lambda^*,0,0)
\left(\begin{array}{l}
\phi_{\lambda^*} \\
\psi_{\lambda^*}\end{array}\right)
\not\in\mathrm{Range}[{\mathscr K}_{(w,v)}(\lambda^*,0,0)].
 \label{3.17}\ees
Suppose for contradiction that there exists $(\phi,\psi)\in X\times X$ such that
$${\mathscr K}_{(w,v)}(\lambda^*,0,0)
\left(\begin{array}{l}\phi \\
\psi\end{array}\right)
={\mathscr K}_{(w,v),\lambda}(\lambda^*,0,0)
\left(\begin{array}{l}
\phi_{\lambda^*} \\
\psi_{\lambda^*}\end{array}\right).$$
A simple calculation yields
$$\left\{\begin{array}{ll}
-D\Delta\psi+\left(\theta_{\lambda^*}F^{\prime}(0)-\mu_{\lm^*}\right)\psi=
\frac{\partial \theta_\lambda}{\partial \lambda}\Big|_{\lambda=\lambda^*}
F^{\prime}(0)\psi_{\lambda^*},  &  x\in\Omega,  \\[1mm]
\psi=0,  &  x\in\partial\Omega.
\end{array}\right.$$
We take the $L^2$-inner product of the above equation with $\psi_{\lambda^*}$
to obtain
 \bes\int_{\Omega}\frac{\partial \theta_\lambda}{\partial \lambda}
\Big|_{\lambda=\lambda^*}F^{\prime}(0)\psi^2_{\lambda^*}dx=0.
\label{3.18}\ees
Recall that $\theta_\lm$ satisfies \eqref{1.6}. It follows that
 \[\sigma_1(d(0),\,2\theta_\lm-\lm;\,1)>\sigma_1(d(0),\,\theta_\lm-\lm;\,1)=0,\]
and the function $w=\frac{\partial \theta_\lambda}{\partial\lambda}$ satisfies
\bess
\left\{\begin{array}{ll}
-d(0)\Delta w+(2\theta_\lm-\lm)w=\theta_\lm>0,~~&x\in\Omega,\\[1mm]
w=0,~~ &x\in\partial\Omega.
\end{array}\right.
 \eess
The strong maximum principle gives $w>0$ in $\Omega$, i.e., $\frac{\partial \theta_\lambda}{\partial
\lambda}\big|_{\lambda=\lambda^*}> 0$ in $\Omega$. Since $\psi_{\lambda^*}> 0$ in $\Omega$, the equation \eqref{3.18} is impossible. Consequently, \eqref{3.17} holds. We can apply the local bifurcation theorem of Crandall
and Rabinowitz (see \cite[Theorem 1.7]{cr1971}) to complete the proof of
Lemma \ref{le36}.
\end{proof}

When $F'(0)>0$, the following lemma gives the global structure of the local
curve $\mathcal{S}_1$ given by (\ref{3.8}) in the $(\lambda,u,v)$ plane.

\begin{Lemma}\label{le37}
Assume that $\mu>D\sigma_1$ and $F'(0)>0$. Then the local curve
$\mathcal{S}_1$ given by (\ref{3.8}) can be extended to a bounded global continuum of positive
solutions to \eqref{ys12} which joins the semitrivial solutions branch
$\Gamma_u$ at $(\lambda^*,\theta_{\lambda^*},0)$.
\end{Lemma}

\begin{proof}
As shown in Lemma \ref{le35}, there exists a continuum $\mathcal{C}\subset
\mathbb{R}\times{\rm int}(\mathcal{P})\times{\rm int}(\mathcal{P})$
of positive solutions of \eqref{ys12} such that $\mathcal{S}_1\subset
\mathcal{C}$ and $\mathcal{C}$ satisfies one of the items (a)-(d). For the
same reason as stated above, it is clear that items (c) and (d) cannot occur.
When $F'(0)>0$, i.e., $\mathbb{F}(0)>0$, Theorem \ref{th11}(1) and (2) show
that \eqref{ys12} has no positive solution if $\lambda$ is too small or large.
Moreover, it follows from Lemma \ref{le22} that any positive solution is
bounded in $W^{2,p}(\Omega)\times W^{2,p}(\Omega)$ for all $p>1$,
provided $\lambda$ is bounded. According to Sobolev embedding theorem, any positive solution is bounded in
$C_0^1(\overline\Omega)\times C_0^1(\overline\Omega)$, and thus (a)
cannot occur. This means that the continuum $\mathcal{C}$ of positive
solutions of \eqref{ys12} must satisfy item (b). Consequently, $\mathcal{C}$
joins the semitrivial solution branch $\Gamma_u$ at
$(\lambda^*,\theta_{\lambda^*},0)$. The proof of Lemma \ref{le37}
is complete.
\end{proof}

Theorem \ref{th11}(4) will follow as a consequence of the results in this
subsection.

\section{Summary}
This paper is devoted to a class of predator-prey systems with prey-taxis,
where the predator species is assumed not only to move around randomly, but
also to be able to direct their movement toward higher concentration of the
prey species. We use some classical tools in local and global bifurcation
theory to establish the existence of some components of coexistence states
bifurcating from the semitrivial steady-state solution, and characterize
the global behavior of these components of coexistence states by applying
some techniques of quasilinear elliptic equations. Moreover, some nearly
optimal non-existence results of coexistence states are also established.

Our mathematical analysis shows that the functional response of the predator
$F(v)$ and the prey-taxis coefficient $\chi(v)$ strongly affect the behavior
of the component of coexistence states. More precisely, our analysis results
lead us to the following findings:

\begin{itemize}[itemsep= -2 pt,topsep = 0 pt, leftmargin = 20 pt]
  \item[(1)]  If the functional response of the predator $F(v)$ satisfies $F'(0)>0$,
  such as Lotka-Volterra type, Holling~type~II and Holling~type~IV response
  functions, then there exists a bounded component of coexistence states
  such that it bifurcates from the semitrivial solution branch
  $\Gamma_{v}=\{(\lambda,0,\omega_\mu):\lambda\in \mathbb{R}\}$ at
  $(\lambda_\mu,0,\omega_\mu)$ and meets the other semitrivial solution
  branch $\Gamma_{u}=\{(\lambda,\theta_\lambda,0):\lambda>d(0)
  \sigma_1\}$ at $(\lambda^*,\theta_{\lambda^*},0)$. Possible bifurcation
  diagram of coexistence states are presented in Figure \ref{fig1}.
  Moreover, our theoretical results show that:
  for weak predator growth rate, the prey-only steady state
  $(0,\omega_\mu)$ is asymptotically stable and two species cannot coexist;
  for intermediate predator growth rate, the prey-only steady
  state $(0,\omega_\mu)$ and the predator-only steady state
  $(\theta_\lambda,0)$ are both unstable and two species can
  coexist; for strong predator growth rate, the predator-only steady state
  $(\theta_\lambda,0)$ is asymptotically stable and two species
  cannot coexist. A special numerical simulation example is presented
 in Figure \ref{fig2}.
\end{itemize}

\begin{figure}[H]
\centerline{
\includegraphics[width=4.0cm,  height=3.5cm]{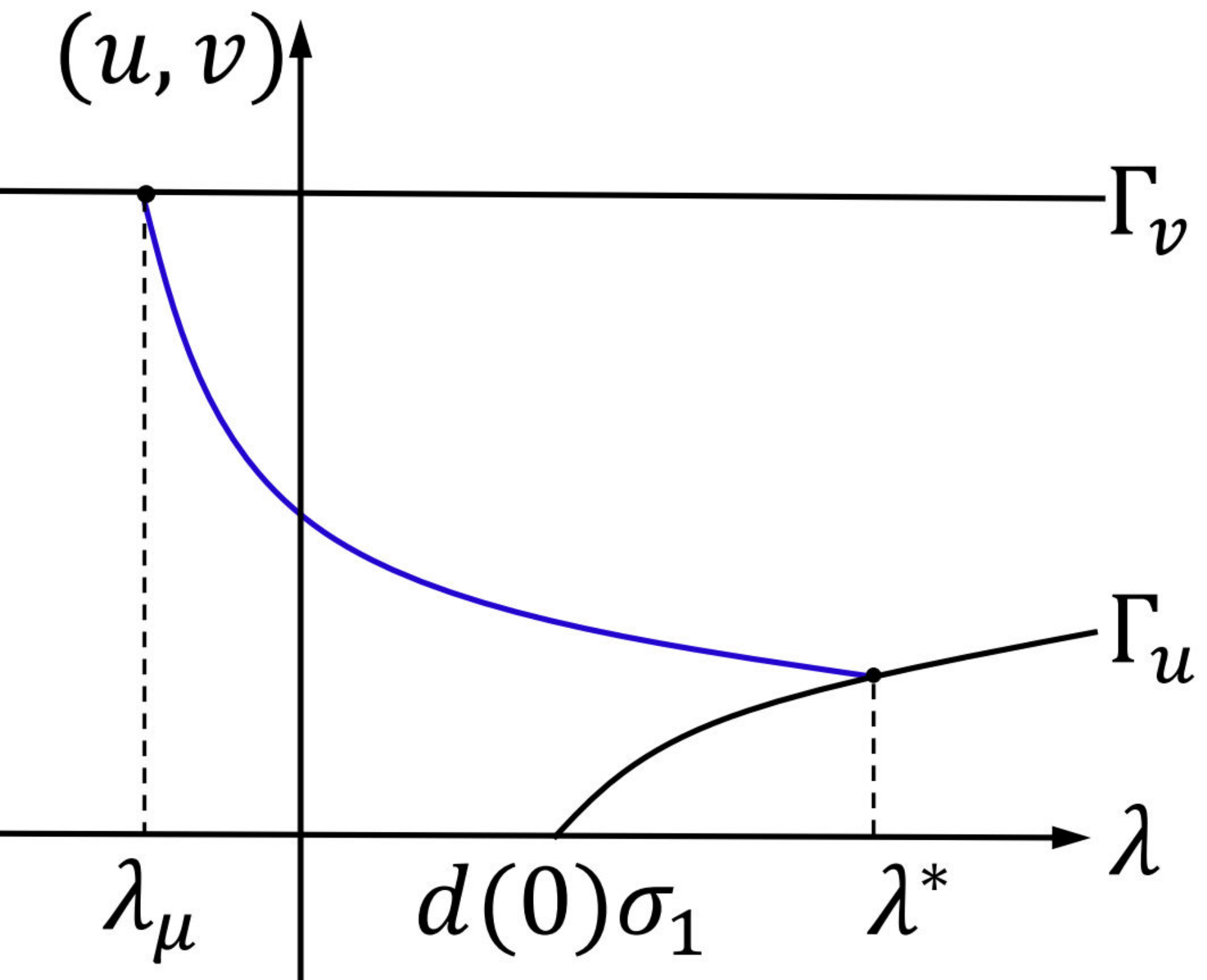}\hspace{10mm}
\includegraphics[width=4.0cm,  height=3.5cm]{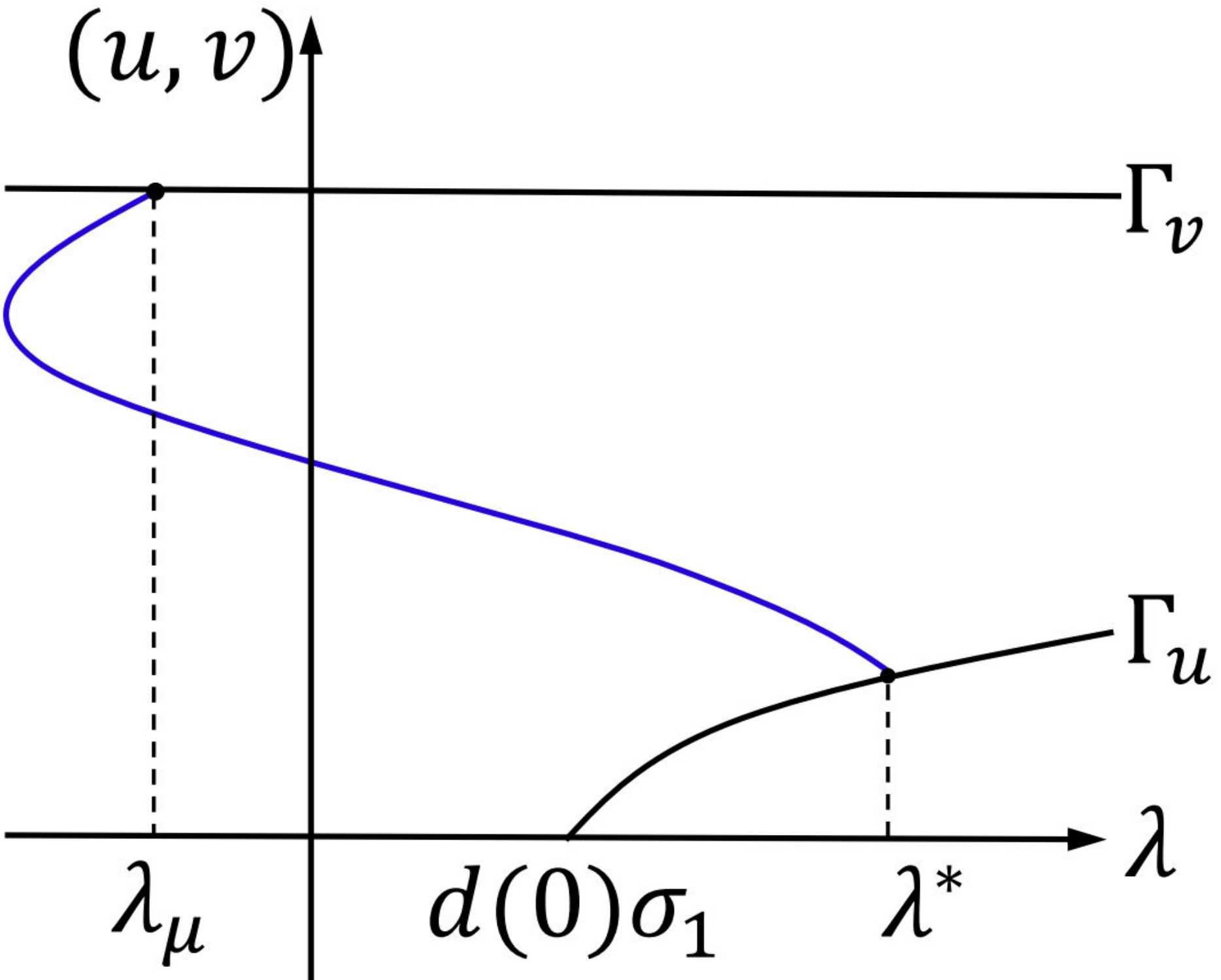}}
%\centerline{\ \ \ (a) \hspace{3.2cm} (b)}
%\vspace{0.1cm}
\caption{\footnotesize Possible bifurcation diagram of coexistence states in the
case $F'(0)>0$.}
\label{fig1}
\end{figure}

\begin{figure}[H]
\centerline{
\includegraphics[width=4.8cm,  height=4cm]{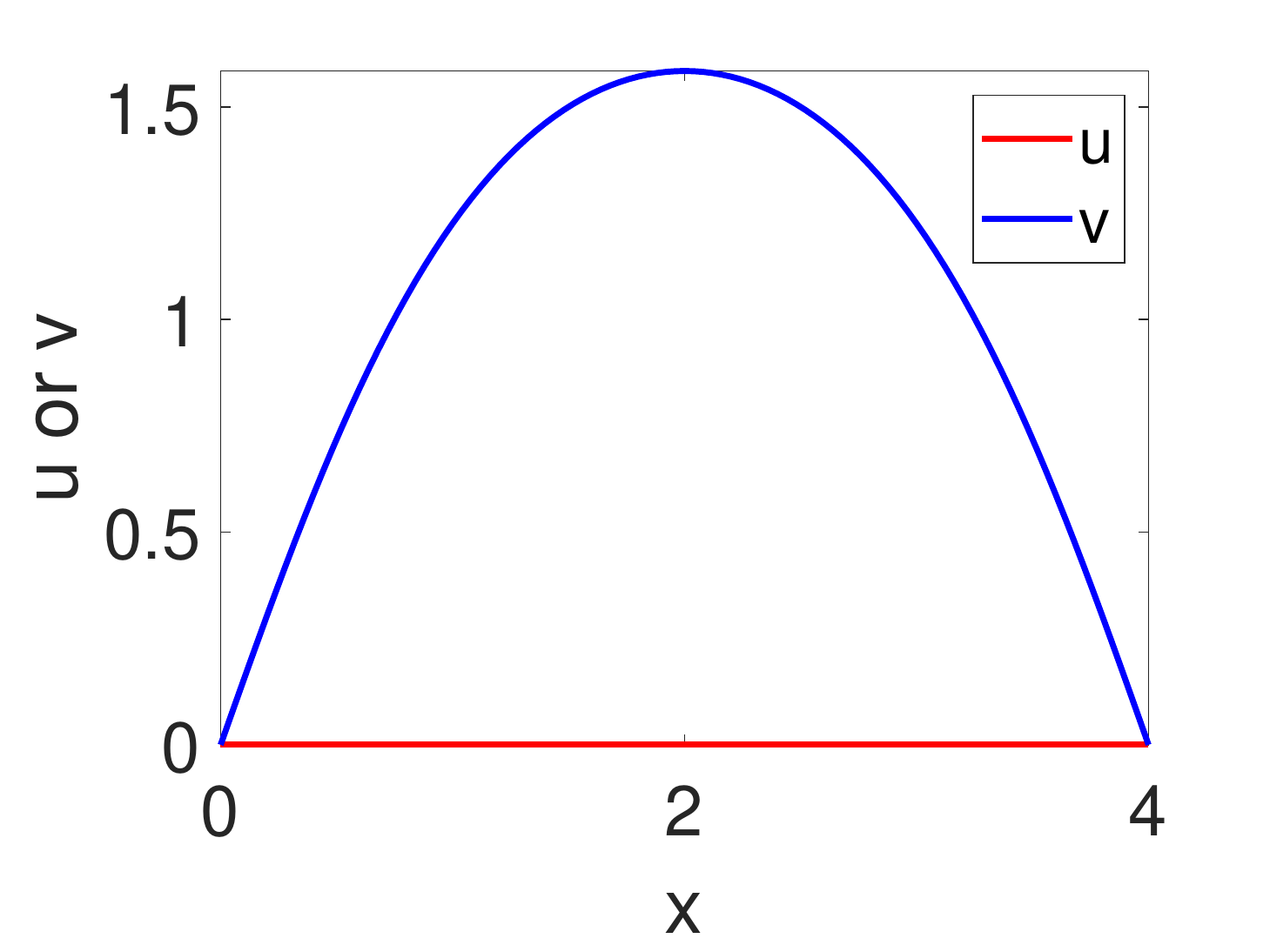}
\includegraphics[width=4.8cm,  height=4cm]{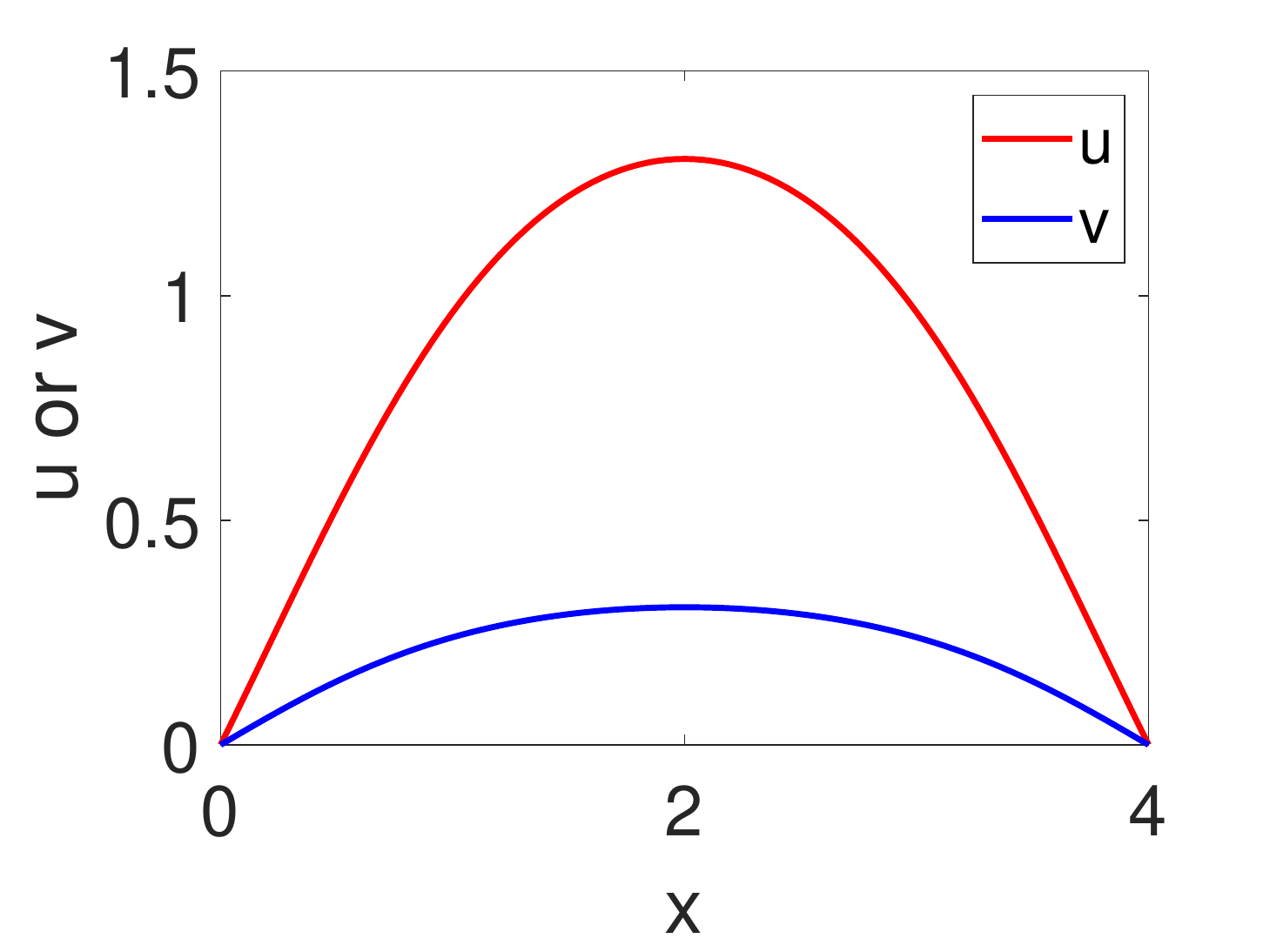}
\includegraphics[width=4.8cm,  height=4cm]{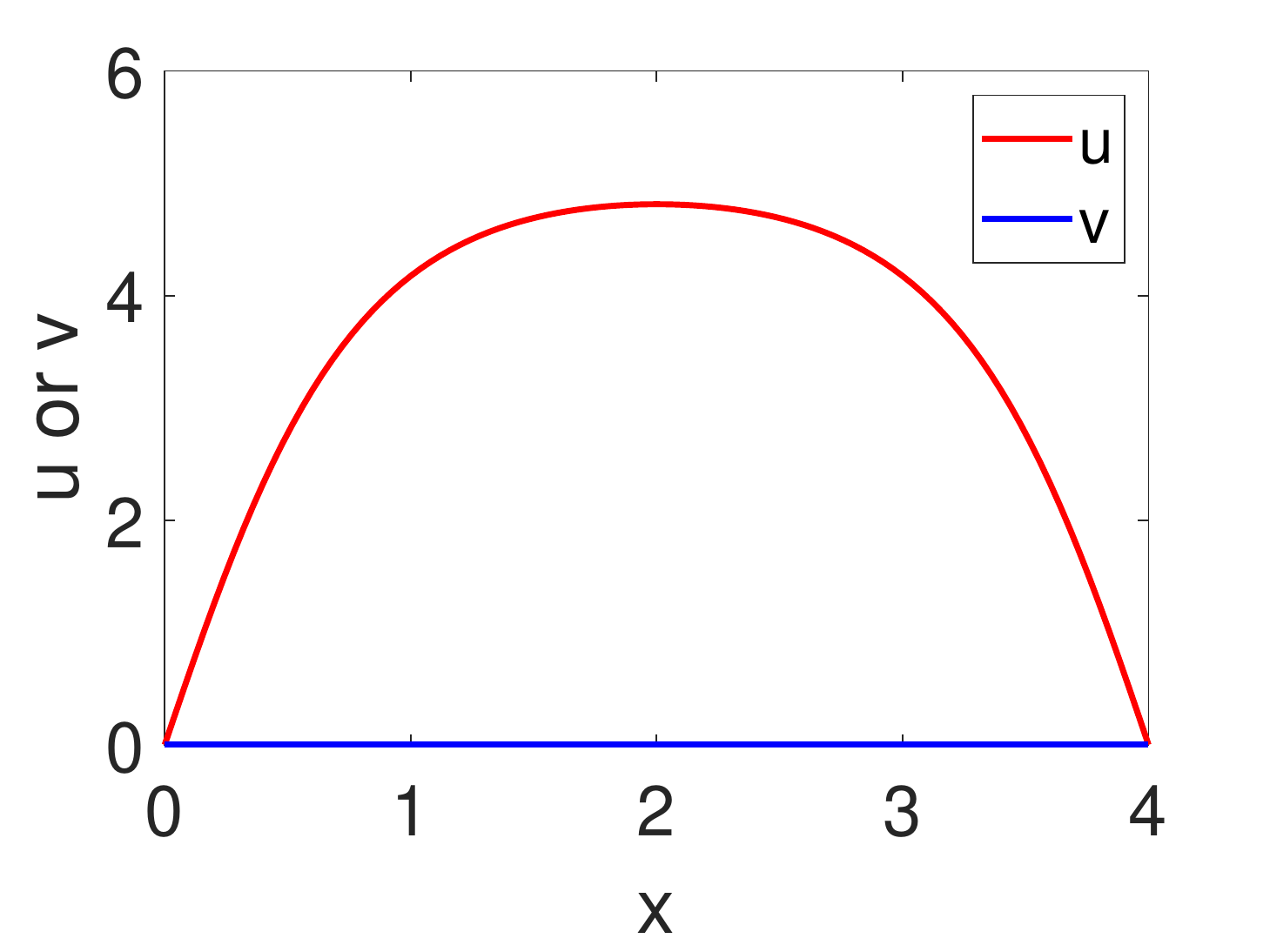}}
\centerline{\small\ \ \ (a)~~$\lambda=-1$ \hspace{2.8cm} (b)~~$\lambda=1.5$
\hspace{2.8cm} (c)~~$\lambda=5$}
\vspace{0.2cm}
\caption{\footnotesize Solution $(u(x,500),v(x,500))$ of
  \eqref{ys11} with $F(v)=v$ (Lotka-Volterra type) in the interval
  $\Omega=(0,4)$, where parameter values are: $d(v)=1$, $\chi(v)=1$, $D=1$,
  $\mu=2$, $\gamma=0.6$, and the initial value are:
  $u_0(x)=0.1+0.1\sin(5x)$, $v_0(x)=0.1+0.1\sin(5x)$.}
\label{fig2}
\end{figure}

\begin{itemize}[itemsep= -2 pt,topsep = 0 pt, leftmargin = 20 pt]
\item[(2)] If the functional response of the predator $F(v)$ satisfies $F'(0)=0$,
  such as Holling~type~III response function, then there exists an unbounded
  component of coexistence states such that it bifurcates from the semitrivial
  solution branch $\Gamma_{v}=\{(\lambda,0,\omega_\mu):\lambda\in
  \mathbb{R}\}$ at $(\lambda_\mu,0,\omega_\mu)$ and extends to infinity
  in positive values of $\lambda$. Possible bifurcation diagram of coexistence
  states are presented in Figure \ref{fig3}. Moreover, our theoretical results show that:   for weak predator growth rate, the prey-only steady state
  $(0,\omega_\mu)$ is asymptotically stable and two species
  cannot coexist;   for intermediate predator growth rate, the prey-only steady state   $(0,\omega_\mu)$ and the predator-only steady state $(\theta_{d(0),
  \lambda},0)$ are both unstable and two species can coexist;
  for strong predator growth rate, the predator-only steady state
  $(\theta_\lambda,0)$ is never a stable, and instead, two species
  can still coexist. A special numerical simulation example is presented in Figure \ref{fig4}.
\end{itemize}

\begin{figure}[H]
\centerline{
\includegraphics[width=4.0cm,  height=3.5cm]{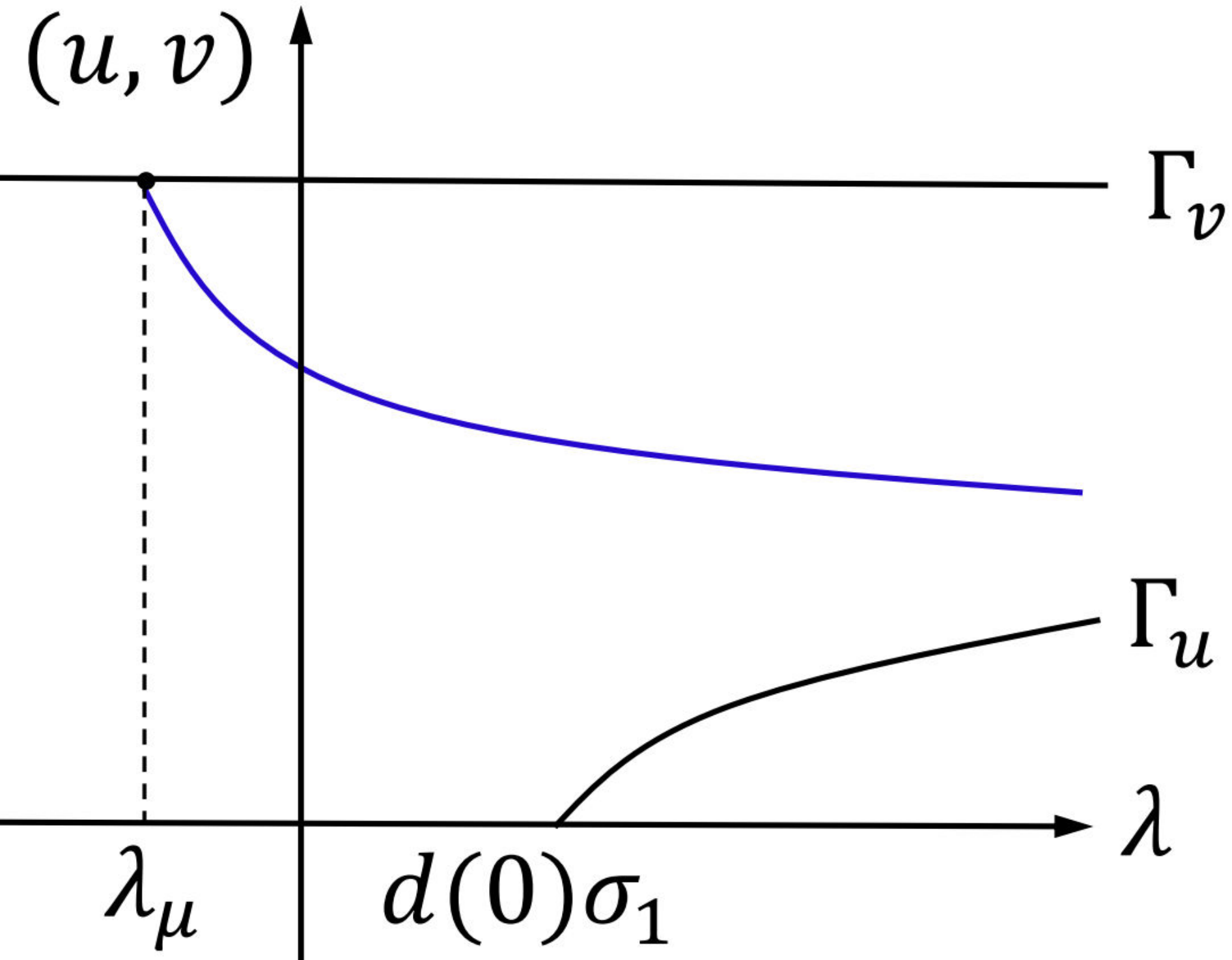}\hspace{10mm}
\includegraphics[width=4.0cm,  height=3.5cm]{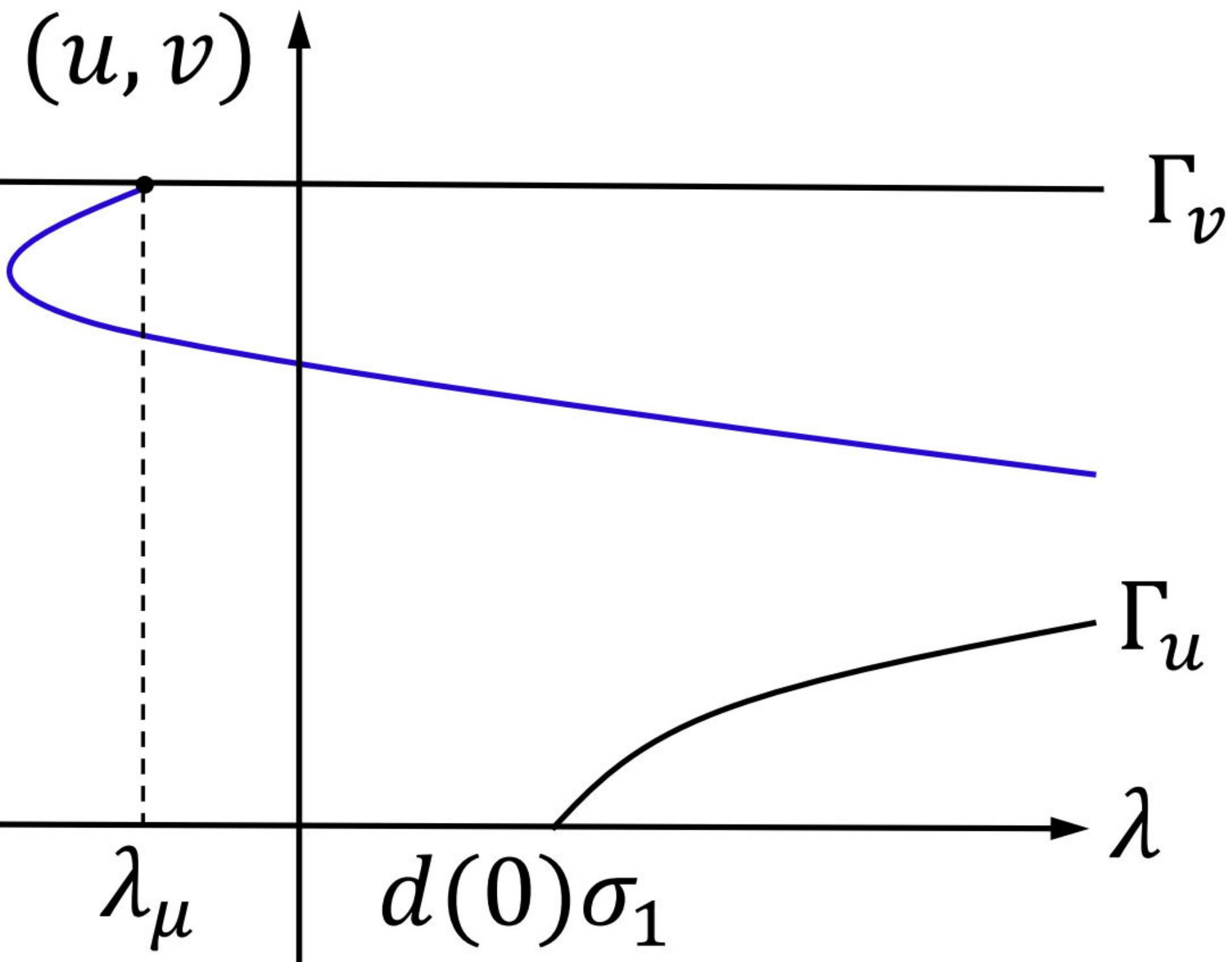}}
%\centerline{\ \ \ (a) \hspace{3.2cm} (b)}
%\vspace{0.1cm}
\caption{\footnotesize Possible bifurcation diagram of coexistence states in the
case $F'(0)=0$.}
\label{fig3}
\end{figure}
\begin{figure}[H]
\centerline{
\includegraphics[width=5.0cm,  height=4.0cm]{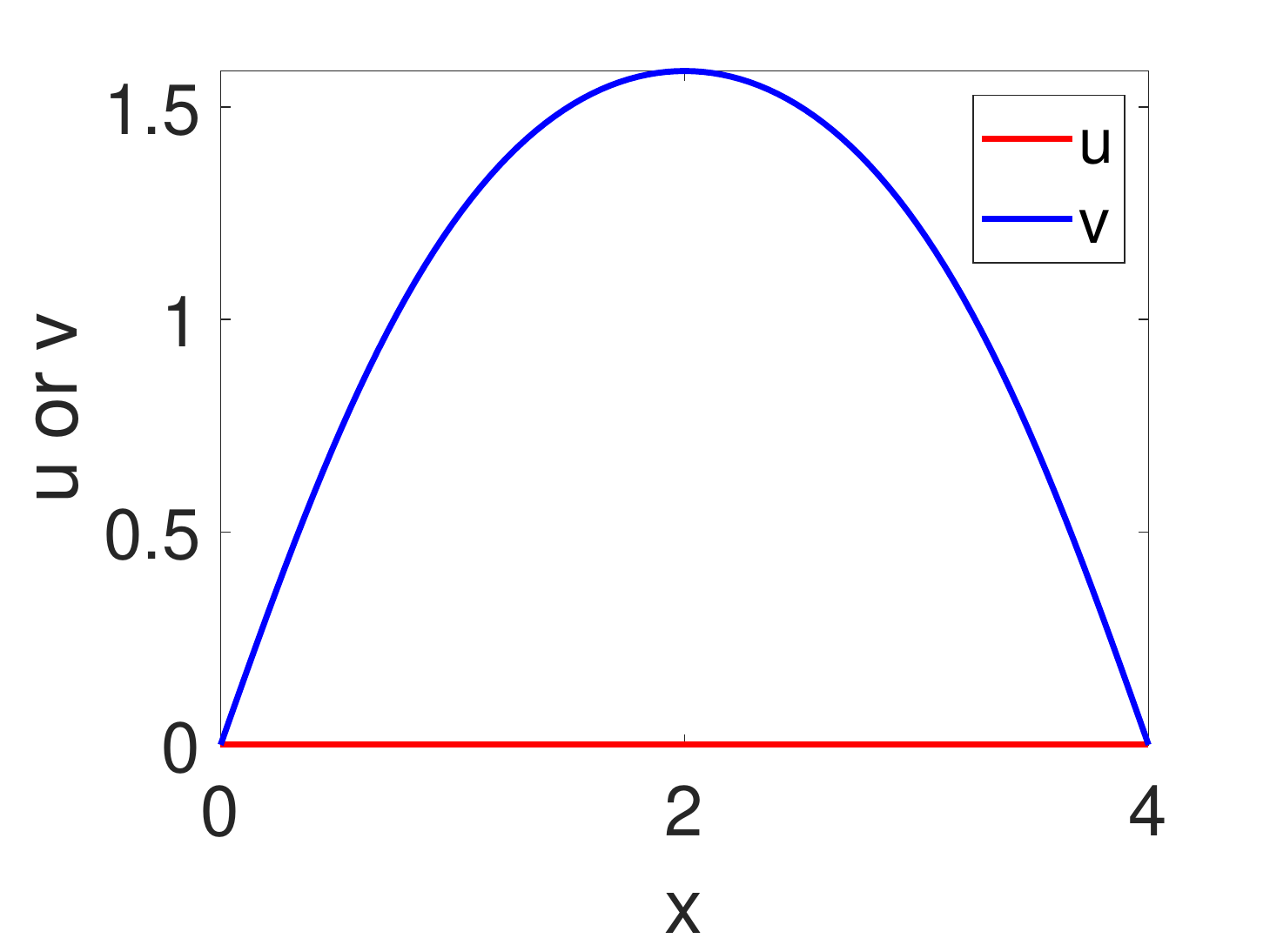}
\includegraphics[width=5.0cm,  height=4.0cm]{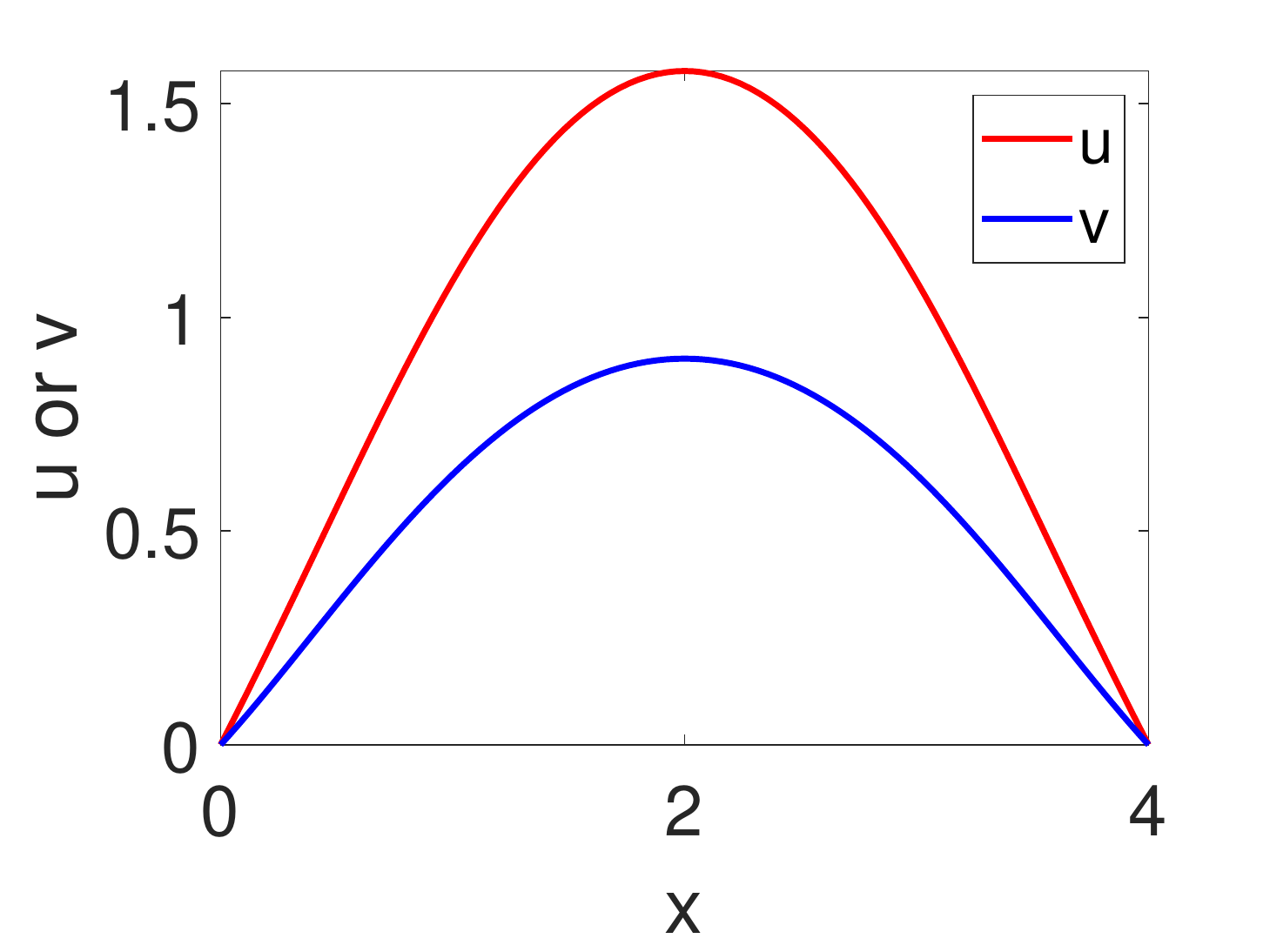}
\includegraphics[width=5.0cm,  height=4.0cm]{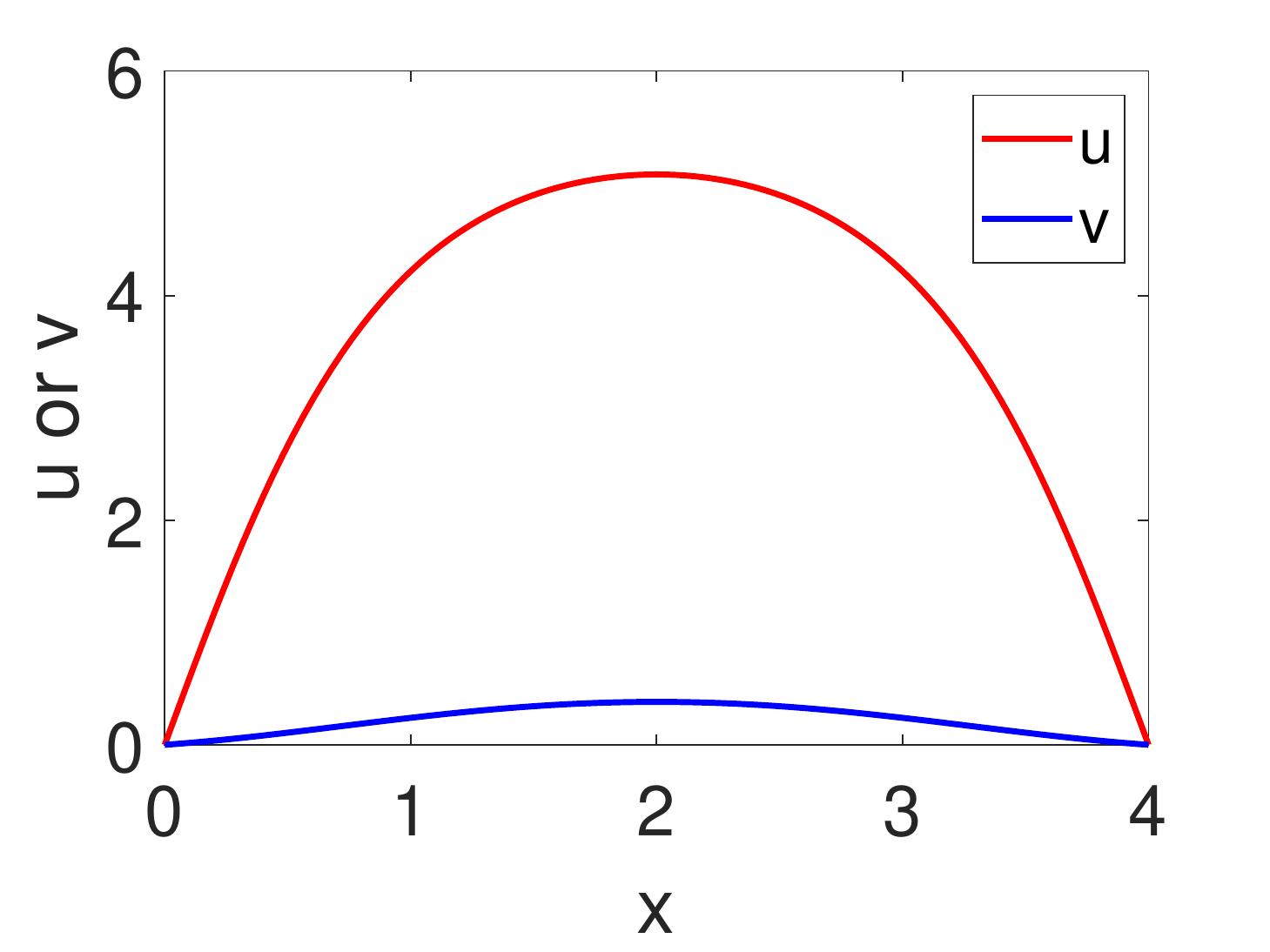}}
\centerline{\small\ \ \ (a)~~$\lambda=-1$ \hspace{2.8cm} (b)~~$\lambda=1.5$
\hspace{2.8cm} (c)~~$\lambda=5$}
\vspace{0.2cm}
\caption{\footnotesize Solution $(u(x,500),v(x,500))$ of
  \eqref{ys11} with $F(v)=\frac{v^2}{1+v^2}$ (Holling III type) in the interval
  $\Omega=(0,4)$, where parameter values are: $d(v)=1$, $\chi(v)=1$, $D=1$,
  $\mu=2$, $\gamma=0.6$, and the initial value are:
  $u_0(x)=0.1+0.1\sin(5x)$, $v_0(x)=0.1+0.1\sin(5x)$.}
\label{fig4}
\end{figure}

\begin{itemize}[itemsep= -2 pt,topsep = 0 pt, leftmargin = 20 pt]
  \item[(3)] If the prey-taxis term is ignored (i.e., $\chi(v)=0$), then the
  bifurcation at $(\lambda_\mu,0,\omega_\mu)$ is supercritical since
  \begin{equation*}
  \lambda^{\prime}_\mu(0)=\left(\int_{\Omega}\Phi_\mu^2dx
  \right)^{-1}\left[\int_{\Omega}d^{\prime}(\omega_\mu)
  \psi_\mu\left|\nabla\Phi_\mu\right|^2dx
  +\int_{\Omega}\Phi_\mu^3dx-\int_{\Omega}\gamma F^{\prime}
  (\omega_\mu)\psi_\mu\Phi_\mu^2dx\right]>0.
  \end{equation*}
However, when the prey-taxis term is considered (i.e., $\chi(v)>0$), the
  bifurcation at $(\lambda_\mu,0,\omega_\mu)$ may be supercritical or
  subcritical, since $\lambda^{\prime}_\mu(0)$ may be positive or negative
  if the appropriate values of parameters in (\ref{ys314}) are selected.
\end{itemize}

There are various interesting questions that deserve further exploration.
Noticing that the bifurcation at $(\lambda_\mu,0,\omega_\mu)$ may be
supercritical or subcritical, this is very likely to lead to the
multiplicity of positive solutions to \eqref{ys12} in a small neighborhood
of $(\lambda_\mu,0,\omega_\mu)$ in $\mathbb{R}\times X\times X$.
Hence, an interesting question is how to establish the multiplicity of
positive solutions to \eqref{ys12}. Compared with the classical
reaction-diffusion equations, the introduction of the prey-taxis term
brings many difficulties in the analytical treatment. While a more
interesting question is how to analyze the asymptotic behavior of positive
solutions to \eqref{ys12} when the prey-taxis coefficient is large.
Additionally, we have shown that positive solutions of \eqref{ys12}
emanate from the semitrivial solution branch
$\Gamma_{v}=\{(\lambda,0,\omega_\mu): \lambda\in \mathbb{R}\}$ if and
only if $\lambda=\lambda_\mu$, which also determines the stability of the
prey-only steady state $(0,\omega_\mu)$, but a further information on
the principal eigenvalue $\lambda_\mu$ still remains open. All these
questions are very interesting and worthwhile to pursue in the future.

\end{document}